\documentclass[review]{elsarticle}
\usepackage{amsmath}
\usepackage[letterpaper,margin=1.0in]{geometry}
\usepackage{graphicx}
\usepackage{amsfonts}
\usepackage{bm}
\usepackage{xcolor}
\usepackage{multirow}
\usepackage{setspace}
\usepackage{multicol}
\usepackage{subcaption}
\usepackage{floatrow}
\newfloatcommand{capbtabbox}{table}[][\FBwidth]
\usepackage{caption}
\usepackage[utf8]{inputenc} 
\usepackage{lineno,hyperref}
\modulolinenumbers[5]

\newcommand{\ipbtobf}[2]{\left\langle#1,#2\right\rangle_{\mathcal{F}_{h}^{\partial}}}

\newcommand{\iipbfb}[2]{\left\langle#1,#2\right\rangle_{\mathcal{F}_h}}

\newcommand{\gbold}{\bm{g}}

\newcommand{\nbold}{\bm{n}}

\newcommand{\ubold}{\bm{u}}

\newcommand{\vbold}{\bm{v}}
\newcommand{\wbold}{\bm{w}}

\newcommand{\xbold}{\bm{x}}

\newcommand{\ipbt}[2]{\left\langle#1,#2\right\rangle_{\partial \mathcal{T}_h}}

\newcommand{\ipbtkf}[2]{\left\langle#1,#2\right\rangle_{F}}
\newcommand{\ipbf}[2]{\left\langle#1,#2\right\rangle_{\mathcal{F}_h}}
\newcommand{\iipbf}[2]{\left\langle#1,#2\right\rangle_{\mathcal{F}_h^i}}

\newcommand{\llbracket}{\left[\!\left[}
\newcommand{\rrbracket}{\right] \! \right]}

\newcommand{\llcurve}{\{\!\!\{}
\newcommand{\rrcurve}{\}\!\!\}}

\begin{document}
\newtheorem{remark}{Remark}
\newtheorem{corollary}{Corollary}
\newtheorem{lemma}{Lemma}
\newtheorem{theorem}{Theorem}
\newtheorem{proof}{Proof}
\def\proof{\par\noindent{\bf Proof.\ } \ignorespaces}
\def\endproof{}
\begin{frontmatter}

\title{Some continuous and discontinuous Galerkin methods and structure preservation for incompressible flows }

\author[mainaddress]{Xi Chen\corref{mycorrespondingauthor}}
\cortext[mycorrespondingauthor]{Corresponding author: xbc5027@psu.edu (Xi Chen)}

\author[mainaddress1]{Yuwen Li}\ead{yuwenli925@gmail.com}

\author[mainaddress2]{Corina Drapaca}\ead{csd12@psu.edu}

\author[mainaddress]{John Cimbala}\ead{jmc6@psu.edu}

\address[mainaddress]{Department of Mechanical Engineering, The Pennsylvania State
University, University Park, PA 16802, USA}

\address[mainaddress1]{Department of Mathematics, The Pennsylvania State University, University Park, PA 16802, USA}

\address[mainaddress2]{Department of Engineering Science and Mechanics, The Pennsylvania State
University, University Park, PA 16802, USA}

\begin{abstract}
In this paper, we present consistent and inconsistent discontinous Galerkin methods for incompressible Euler and Navier-Stokes equations with the kinematic pressure, Bernoulli function and EMAC function. Semi- and fully discrete energy stability of the proposed dG methods are proved in a unified fashion.  Conservation of total energy, linear and angular momentum is discussed with both central and upwind fluxes. Numerical experiments are presented to demonstrate our findings and compare our schemes with conventional schemes in the literature in both unsteady and steady problems. Numerical results show that global conservation of the physical quantities may not be enough to demonstrate the performance of the schemes, and our schemes are competitive and able to capture essential physical features in several benchmark problems.
\end{abstract}



\end{frontmatter}

\section{Introduction}
There has been extensive research on Galerkin methods for incompressible flows, see, e.g., \cite{Temam1984,GR1986,Hesthaven07,layton2008introduction,Ern2012} and references therein. Since the pressure is usually viewed as a Lagrange multiplier, it could be naturally decoupled from the velocity by restricting the problem on the divergence-free subspace \cite{girault2012finite,allaire2007numerical} at the continuous level. Since the discrete divergence-free subspace is often not contained in the continuous one, classical $H^{1}$-conforming methods suffer from a loss of numerical accuracy as the velocity error is  affected by the pressure approximation, especially under small viscosity, which is known in the literature as a lack of pressure robustness \cite{john2017divergence}.
In the current paper, we focus on ways to reduce the influences of pressure approximations on velocity approximations. With the help of a grad-div stabilization term \cite{franca1988two}, one may increase pressure robustness although the mass conservation is violated, see, e.g.,  \cite{john2017divergence,olshanskii2004grad,olshanskii2009grad,lube2002stable,olshanskii2002low,case2011connection,linke2011convergence,jenkins2014parameter}. Another remedy is to use $H(\text{div})$-conforming methods that completely remove  pressure influence on the velocity approximation through a carefully chosen finite element  pair~\cite{RT1977,BDM1985,boffi2013mixed}. In this case, the numerical velocity is pointwise divergence-free \cite{john2017divergence,guzman2016h,schroeder2018towards,zhang2005new,falk2013stokes,guzman2014conforming,lehrenfeld2016high}. Moreover, in the vorticity-stream formulation, one could automatically enforce the divergence-free constraint at the PDE level when the spatial dimension is two \cite{GR1986,liu2000high,cai2019high}.
Finally, for discontinuous Galerkin (dG) methods, a common technique for pressure robustness is penalizing the jump of the velocity normal component \cite{guzman2016h,akbas2018analogue,joshi2016post,gresho1990theory}. An  interesting result is reported in \cite{schroeder2017stabilised}, where an element-wise grad-div penalization was used on tensor product meshes for a non-isothermal flow, and an improvement of the mass conservation is observed for an inf-sup stable element pair of equal order. A discrete inf-sup condition involving the pressure jump is constructed in \cite{Ern2012} for both the steady incompressible Stokes and Navier-Stokes equations. In \cite{linke2014role}, reconstruction with the lowest-order divergence-free Raviart-Thomas velocity is used to recover the $L^{2}$-orthogonality of the discretely divergence-free velocities and irrotational vector fields. Other relevant dG methods can be found in e.g., \cite{cockburn2005locally,cesmelioglu2017analysis,cockburn2007note}.

In this paper, we generalize our previous work \cite{xi2020} and present two new classes of dG methods for incompressible Euler and  Navier-Stokes equations. The unified numerical framework in \cite{xi2020} is designed for incompressible Navier-Stokes equations with the kinematic pressure $p$.  Here, this approach is further extended to cover the Bernoulli function $p+\frac{\theta}{2}\ubold\cdot\ubold$, and EMAC function $p-\frac{\theta}{2}\ubold\cdot\ubold$. The Bernoulli function is roughly speaking the “force” one feels when one faces into the wind, an important physical quantity rarely considered as a direct output in computational fluids. The EMAC function was first proposed in \cite{charnyi2017conservation} for conserving energy, linear and angular momentum by $H^{1}$-conforming finite element methods. We refer to \cite{charnyi2019efficient,kundufluid,olshanskii2020longer,schroeder2017pressure} and references therein for more discussions about these two functions.

In addition,
we shall show that both classes of dG methods achieve the same numerical velocity field when the velocity is $H(\text{div})$-conforming, and we shall analyze their conservation property for energy, linear momentum and angular momentum with both central and upwind fluxes. To the best of our knowledge, such frameworks and conservation analysis  for $H(\text{div})$-conforming and dG methods do not exist in the literature. From our dG framework, it is also shown that global energy conservation for inviscid flows is generally not worth pursuing at the expense of pressure robustness.
Nevertheless, all proposed methods in this framework are proved to be energy stable in the  fully-discrete level,  which is an important property in the simulation of turbulent flows (cf.~\cite{fehn2018robust}).

\section{Preliminaries}
Let $\mathcal{T}_h$ denote a conforming and shape-regular simplicial mesh on a  polyhedral domain ${\Omega}$ in $\mathbb{R}^d$ with $d\in\{2,3\}$. Let $\mathcal{F}_{h}$ denote the collection of faces of $\mathcal{T}_h$,  $\mathcal{F}_h^{i}$ the set of interior faces, and $\mathcal{F}_h^{\partial}$ the set of boundary faces. For $K\in\mathcal{T}_h$ and $F\in\mathcal{F}_h$, we use $h_{K}$ to denote the diameter of $K$, $h_F$ the diameter of $F\in\mathcal{F}_h$. For any $(d-1)$-dimensional set $\Sigma$, let $\langle\cdot,\cdot\rangle_\Sigma$ denote the $L^2$ inner product on $\Sigma$, and 
\begin{align*}
&\ipbt{\cdot}{\cdot}:= \sum_{K \in \mathcal{T}_h} \langle\cdot,\cdot\rangle_{\partial K},\quad\langle\cdot,\cdot\rangle_{\partial\mathring{\mathcal{T}}_h}:=\sum_{K\in\mathcal{T}_h}\langle\cdot,\cdot\rangle_{\partial K\backslash\partial\Omega},\\
&\ipbf{\cdot}{\cdot}:= \sum_{F \in \mathcal{F}_h} \ipbtkf{\cdot}{\cdot},\quad
\langle\cdot,\cdot\rangle_{\mathcal{F}^i_h}:= \sum_{F\in\mathcal{F}^i_h} \ipbtkf{\cdot}{\cdot}.
\end{align*}
Given $F\in\mathcal{F}_h^{i}$, we fix a unit normal $\bm{n}_F$ to $F$, which points from one element $K^+$ to the other element $K^-$ on the other side. The jump and averaging operators are defined as
\begin{align*}
\llbracket \phi \rrbracket|_F&= \phi|_{K^+} - \phi|_{K^-}, \qquad \llbracket \phi \nbold \rrbracket|_F = \phi|_{K^+} \bm{n}_F-\phi|_{K^-} \bm{n}_F, \qquad \llcurve \phi \rrcurve|_F= \frac{1}{2} \left( \phi|_{K^+} + \phi|_{K^-} \right), \\
\llbracket \vbold \rrbracket|_F&= \vbold|_{K^+} - \vbold|_{K^-}, \qquad \llbracket \vbold \otimes \nbold \rrbracket|_F= \vbold|_{K^+} \otimes \bm{n}_F- \vbold|_{K^-} \otimes \bm{n}_F, \qquad  \llcurve \vbold \rrcurve|_F= \frac{1}{2} \left( \vbold|_{K^+} + \vbold|_{K^-}\right),
\end{align*}
where $\phi$ and $\vbold$ are arbitrary scalar- and vector-valued functions, respectively. For a boundary face $F\in\mathcal{F}_h^{\partial}$ which is contained in a single element $K\in\mathcal{T}_h$, we further assume that $\bm{n}_F$ is the outward pointing normal to $\partial\Omega$ and define
\begin{align*}
\llbracket \phi \rrbracket|_F&= \phi|_K, \qquad \llbracket \phi \nbold \rrbracket|_F= \phi|_K\bm{n}_F, \qquad \llcurve \phi \rrcurve|_F= \phi|_K, \\
\llbracket \vbold \rrbracket|_F&= \vbold|_K, \qquad \llbracket \vbold\otimes \nbold \rrbracket|_F= \vbold|_K\otimes \bm{n}_F, \qquad  \llcurve \vbold \rrcurve|_F= \vbold|_K.
\end{align*}
Throughout the rest of this paper, we use $\bm{n}\in\prod_{F\in\mathcal{F}_h}\mathbb{R}^d$ to denote the piecewise constant vector defined on the skeleton $\mathcal{F}_h$ such that $\bm{n}|_F:=\bm{n}_F$ for all $F\in\mathcal{F}_h.$ 
Let $\mathcal{P}_j \left(K \right)$ denote the space of polynomials of degree at most $j$ (with $j$ a non-negative integer). We shall make use of the following function spaces
\begin{align*}
&L^{2}_{0}(\Omega) = \left\{q \in {L}^{2}(\Omega):  \int_{\Omega}q=0\right\},\\
&[H^{m}(\mathcal{T}_h)]^d = \left\{\vbold  \in [{L}^{2}(\Omega)]^d: \vbold|_{K} \in[{H}^{m}(K)]^d,~\forall K \in \mathcal{T}_h \right\}, \\
&\bm{V}_h:= \left\{ \bm{v}_h\in[{L}^{2}(\Omega)]^d: \bm{v}_h |_{K} \in  [\mathcal{P}_{k+1} \left( K \right)]^{d}, \forall K \in \mathcal{T}_h \right\},\\[1.5ex]
&\bm{V}_h^{\text{div}}:= \left\{ \bm{v}_h \in H\left(\text{div}; \Omega \right):  \bm{v}_h |_{K} \in  \bm{\mathcal{Q}}_k \left( K \right),~\forall K \in \mathcal{T}_h~\text{and}~\bm{v}_h\cdot\bm{n}|_{\partial\Omega}=0\right\},\\[1.5ex]
&{Q}_h:= \left\{q_h\in L^{2}_{0}(\Omega): q_{h} |_{K} \in \mathcal{P}_k \left( K \right), \forall K \in \mathcal{T}_h \right\}.
\end{align*}
where $\bm{\mathcal{Q}}_k(K):=[\mathcal{P}_{k}(K)]^d+\mathcal{P}_{k}(K)\bm{x}$ or $\bm{\mathcal{Q}}_k(K):=[\mathcal{P}_{k+1}(K)]^d$ corresponding to the Raviart--Thomas (RT) \cite{RT1977} or Brezzi--Douglas--Marini (BDM) \cite{BDM1985} shape function space, respectively.

The rest of this paper is organized as follows. In Section \ref{euler}, we first present the unified framework and then prove the semi- and fully-discrete stability of the general scheme for the incompressible Euler equations in the general form, then we briefly extend our schemes to the incompressible Navier-Stokes equations. In Section \ref{conservation_property}, we analysis the conservation properties of the $H(\text{div})$-conforming and dG methods. In Section \ref{secNE}, we perform numerical experiments to confirm our findings in Section \ref{conservation_property} and test our dG schemes in both unsteady and steady situations, and compare the results with conventional schemes in the literature. Finally we conclude our paper in Section \ref{conclusion}.
\section{Incompressible Euler Equations} \label{euler}

Consider the incompressible Euler equations in the following form
\begin{equation}\label{incompressible_Euler}
    \begin{aligned}
\partial_t\ubold + \nabla \cdot \left( \ubold \otimes \ubold + P\mathbb{I} \right) - \theta\ubold\cdot\nabla\ubold^{T}= 0,  \qquad & \text{in} \quad \left(0, T\right]\times \Omega,\\
\nabla \cdot \ubold = 0, \qquad & \text{in} \quad \left(0, T\right]\times \Omega,\\  
\ubold\cdot\nbold=\bm{0},\qquad& \text{on} \quad \left(0, T\right] \times \partial\Omega,\\
\ubold(0,\xbold)=\ubold_0(\xbold),\qquad& \text{in} \quad \Omega, 
\end{aligned}
\end{equation}
where $P$ is defined as  
\begin{equation}\label{phat}
    P:=p+\frac{\theta}{2}\bm{u}\cdot\bm{u},
\end{equation}
with $\theta$ being an arbitrary constant. Note that $P$ unifies the kinematic pressure ($\theta=0$), Bernoulli function ($\theta=1.0$) and EMAC function ($\theta=-1.0$). Although the EMAC function is not a physical quantity, its corresponding formulation could be used to calculate the numerical velocity.

Let $(\cdot,\cdot)$ denote the $L^2$ inner product on $\Omega$ and $\nabla_h$ the broken gradient with respect to $\mathcal{T}_h$. For the form with $\theta=0$, we follow the derivation in \cite{xi2020} and obtain a semi-discrete scheme: Find unknowns  $\left(\bm{u}_h(t),P_h(t) \right)\in\bm{V}_h\times Q_h$ for each time $t\in(0,T]$ such that
\begin{subequations}\label{mixed}
    \begin{align}
&(\partial_t\ubold_h,\vbold_h)- (\ubold_h \otimes \ubold_h,\nabla_h \vbold_h)- (P_h,\nabla_h \cdot \vbold_h) + \ipbt{\widehat{\bm{\sigma}}_{h} \nbold}{\bm{v}_h} \\[1.5ex] 
\nonumber & -\frac{1}{2} (\left(\nabla_h \cdot \bm{u}_h \right) \ubold_h,\vbold_h) + \frac{1}{2}  \langle\bm{u}_h,\nbold \llcurve \ubold_h \cdot \vbold_h\rrcurve\rangle_{\partial\mathring{\mathcal{T}}_h}+d_h(\bm{u}_h,\bm{v}_h) =0,\\[1.5ex]
&(\nabla_h\cdot \ubold_h,q_h)-\iipbfb{\llbracket \ubold_h \rrbracket\cdot\nbold}{\llcurve q_h \rrcurve} = 0,
\end{align}
\end{subequations}
for all $\left(\bm{v}_h,q_h\right)\in\bm{V}_h\times Q_h$. 
The penalty term $d_h(\bm{u}_h,\bm{v}_h)$ is defined as
\begin{align}
d_h(\bm{u}_h,\bm{v}_h)=\gamma\bigg(\sum_{F\in\mathcal{F}_h}h_F^{-1}\langle\llbracket \ubold_h \rrbracket\cdot\nbold_F,\llbracket \vbold_h \rrbracket\cdot\nbold_F\rangle_F+(\nabla_h\cdot\ubold_h,\nabla_h\cdot\vbold_h)\bigg),\label{d_h}
\end{align}
for some constant $\gamma>0$, which is used to increase pressure robustness. The numerical flux is 
\begin{align*}
\widehat{\bm{\sigma}}_h &= \llcurve \ubold_h \rrcurve \otimes \llcurve \ubold_h \rrcurve + \llcurve P_h \rrcurve \mathbb{I}+\zeta\left| \llcurve\ubold_h\rrcurve \cdot \nbold\right| \llbracket \ubold_h \otimes \nbold \rrbracket,
\end{align*}
where $\zeta=\{\zeta_F\}_{F\in\mathcal{F}_h}$ are user-specified piecewise non-negative constants controlling the amount of numerical dissipation. Using element-wise integration by parts, the scheme \eqref{mixed} becomes
\begin{subequations}
\begin{align}
&(\partial_t\ubold_h,\vbold_h) +c^0_h \left(\ubold_h; \ubold_h, \vbold_h \right)- b_h \left( \vbold_h, P_h \right)+d_h(\ubold_h,\vbold_h)=0,\quad\forall\bm{v}_h\in\bm{V}_h,\\
& b_h \left(\ubold_h, q_h \right) = 0,\quad \forall q_h\in Q_h,
\end{align}
\end{subequations}
where 
\begin{align}
&c^0_{h} \left(\bm{\beta}_h; \vbold_h, \wbold_h \right) := (\bm{\beta}_h \cdot \nabla_h \vbold_h,\wbold_h)+ \frac{1}{2} (\left(\nabla_h \cdot \bm{\beta}_h \right) \vbold_h,\wbold_h)- \iipbf{ \left( \llcurve\bm{\beta}_h\rrcurve \cdot \nbold \right) \llbracket \vbold_h \rrbracket}{\llcurve \wbold_h \rrcurve}\nonumber\\
&\quad-\frac{1}{2}\iipbf{\llbracket \beta_h \rrbracket\cdot\nbold}{\llcurve \vbold_h\cdot\wbold_h \rrcurve}+ \ipbf{\zeta \left| \llcurve\bm{\beta}_h\rrcurve \cdot \bm{n} \right|\llbracket \vbold_h \rrbracket}{\llbracket \wbold_h \rrbracket},\label{ch}\\
&b_h(\bm{v}_h, q_h):= (\nabla_h \cdot \vbold_h,q_h)-\iipbfb{\llbracket \vbold_h \rrbracket\cdot\nbold}{\llcurve q_h \rrcurve}\label{bh}.
\end{align} 

However, the form $c_h^0$ is inconsistent when $\theta\neq0$.
In general, for $\theta\in\{0,1,-1\}$, our semi-discrete scheme for \eqref{incompressible_Euler} seeks $(\bm{u}_h(t),P_h(t))\in\bm{V}_h\times Q_h$ such that
\begin{subequations}\label{compact}
\begin{align}
&(\partial_t\ubold_h,\vbold_h) +{c}_h \left(\ubold_h; \ubold_h, \vbold_h \right)- b_h \left( \vbold_h, P_h \right)+d_h(\ubold_h,\vbold_h)=0,\quad\forall\bm{v}_h\in\bm{V}_h,\label{compact1}\\[1.5ex]
& b_h \left(\ubold_h, q_h \right) = 0,\quad \forall q_h\in Q_h.\label{compact2}
\end{align}
\end{subequations}
The term ${c}_h \left(\ubold_h; \ubold_h, \vbold_h \right)$ coincides with $c^0_h \left(\ubold_h; \ubold_h, \vbold_h \right)$ if $\theta=0$ and will be specified later if $\theta\neq0$. In particular, we want $c_h \left(\ubold_h; \ubold_h, \vbold_h \right)\approx c^0_h \left(\ubold_h; \ubold_h, \vbold_h \right)-\theta E_h,$ where $E_h$ approximates $(\ubold\cdot\nabla\ubold^{T},\vbold).$

A key observation is that $c_h^0(\bm{\beta}_h;\bm{v}_h,\bm{w}_h)$ in \eqref{ch} is an approximation to $(\bm{\beta}\cdot\nabla\vbold,\wbold)$. Therefore for the new term $\ubold\cdot\nabla\ubold^{T}$ in \eqref{incompressible_Euler}, its variational counterpart $(\ubold\cdot\nabla\ubold^{T},\vbold)=(\vbold\cdot\nabla\ubold,\ubold)$ \footnote{Here the property of matrix transposition is used.}  could be discretized as 
\begin{align*}
(\bm{u}\cdot\nabla\bm{u}^{T},\bm{v})\approx(\vbold_h \cdot \nabla_h \ubold_h,\ubold_h)+ \frac{1}{2} (\left(\nabla_h \cdot \vbold_h \right) \ubold_h,\ubold_h)- \iipbf{ \left( \llcurve\vbold_h\rrcurve \cdot \nbold \right) \llbracket \ubold_h \rrbracket}{\llcurve \ubold_h \rrcurve}\nonumber-\frac{1}{2}\iipbf{\llbracket \vbold_h \rrbracket\cdot\nbold}{\llcurve \ubold_h\cdot\ubold_h \rrcurve}.
\end{align*}
Here we discard the term $\ipbf{\zeta \left| \llcurve\vbold_h\rrcurve \cdot \bm{n} \right|\llbracket \ubold_h \rrbracket}{\llbracket \ubold_h \rrbracket}$ in $c_h^0(\bm{v}_h;\bm{u}_h,\bm{u}_h)$ as it lacks physical meaning.

\begin{itemize}
\item Scheme \textsf{dG1}

Based on the above analysis, we introduce
\begin{equation}\label{ch_rotational}
    \begin{aligned}
&c^1_{h} \left(\ubold_h; \ubold_h, \vbold_h \right):= (\ubold_h \cdot \nabla_h \ubold_h,\vbold_h)+ \frac{1}{2} (\left(\nabla_h \cdot \ubold_h \right) \ubold_h,\vbold_h)- \iipbf{ \left( \llcurve\ubold_h\rrcurve \cdot \nbold \right) \llbracket \ubold_h \rrbracket}{\llcurve \vbold_h \rrcurve}\\
&\quad-\frac{1}{2}\iipbf{\llbracket \ubold_h \rrbracket\cdot\nbold}{\llcurve \ubold_h\cdot\vbold_h \rrcurve}+ \ipbf{\zeta \left| \llcurve\ubold_h\rrcurve \cdot \bm{n} \right|\llbracket \ubold_h \rrbracket}{\llbracket \vbold_h \rrbracket}+\theta\iipbf{ \left( \llcurve\vbold_h\rrcurve \cdot \nbold \right) \llbracket \ubold_h \rrbracket}{\llcurve \ubold_h \rrcurve}\\
&\quad+\frac{\theta}{2}\iipbf{\llbracket \vbold_h \rrbracket\cdot\nbold}{\llcurve \ubold_h\cdot\ubold_h \rrcurve} -\theta(\vbold_h \cdot \nabla_h \ubold_h,\ubold_h)-\frac{\theta}{2} (\left(\nabla_h \cdot \vbold_h \right) \ubold_h,\ubold_h).
\end{aligned} 
\end{equation}
The corresponding scheme \eqref{compact} with $c_h=c_h^1$ is denoted as \textsf{dG1}.
The advantage is that the positivity of $c^{1}_{h}$ is closely related to the positive structure of $c^0_{h}$. Unfortunately,  $c^{1}_{h}$ is inconsistent due to the two terms $\frac{\theta}{2} (\left(\nabla_h \cdot \vbold_h \right) \ubold_h,\ubold_h)$ and $\frac{\theta}{2}\iipbf{\llbracket \vbold_h \rrbracket\cdot\nbold}{\llcurve \ubold_h\cdot\ubold_h \rrcurve}$. 
\begin{remark}
The inconsistent  \emph{\textsf{dG1}} scheme is the starting point of our framework and is conveniently derived from an existing discrete convective form without redoing tedious integration-by-parts.
Later we shall develop another dG method to enforce consistency.  Furthermore, derivations of the embedded $H(\emph{div})$- and $H^{1}$-conforming methods become transparent by imposing stronger  continuity on dG spaces.
\end{remark}

Let us define the kernel of $b_h$ as
\begin{align*}
\bm{Z}_{h}:=\{\bm{v}_h\in\bm{V}_h:b_h(\bm{v}_h, q_h) = 0,~\forall q_h\in Q_h\}.
\end{align*}
Then restricting  \eqref{compact1} with $c_h=c_h^1$ to $\bm{Z}_{h}$ yields
\begin{equation}\label{eqnZh}
(\partial_t\ubold_h,\vbold_h) +c^{1}_{h} \left(\ubold_h; \ubold_h, \vbold_h \right)+d_h(\ubold_h,\vbold_h)=0,\quad\forall\bm{v}_h\in\bm{Z}_{h}.
\end{equation}
Using Cauchy-Schwarz inequality, we have
\begin{align*}
&(\left(\nabla_h \cdot \vbold_h \right) \ubold_h,\ubold_h)\leq\|\nabla_h\cdot\vbold_h\|_{L^{2}(\Omega)}\|\ubold_h\cdot\ubold_h\|_{L^{2}(\Omega)},\\
&\iipbf{\llbracket \vbold_h \rrbracket\cdot\nbold}{\llcurve \ubold_h\cdot\ubold_h \rrcurve}\leq\sum_{F\in\mathcal{F}_h^{i}}\|\llbracket \vbold_h \rrbracket\cdot\nbold_F\|_{L^{2}(F)}\|\llcurve \ubold_h\cdot\ubold_h \rrcurve\|_{L^{2}(F)},\quad\forall \vbold_h\in\bm{Z}_{h}.
\end{align*}
It follows that one could choose sufficiently large penalizing parameter $\gamma$ in $d_h$ to control $\|\nabla_h\cdot\vbold_h\|_{L^{2}(\Omega)}$, $\|\llbracket \vbold_h \rrbracket\cdot\nbold_F\|_{L^{2}(F)}$, and in turn to minimize the inconsistent terms $(\left(\nabla_h \cdot \vbold_h \right) \ubold_h,\ubold_h)$,  $\iipbf{\llbracket \vbold_h \rrbracket\cdot\nbold}{\llcurve \ubold_h\cdot\ubold_h \rrcurve}$, which are variational crimes. Therefore, the consistency of \textsf{dG1} is weakly enforced. We will show through numerical experiments in Section \ref{secNE} that one may choose $\gamma$ properly to achieve high numerical accuracy.
We note that the inconsistent terms will have an effect on the approximation of $P$ as the penalization is enforced in $\bm{Z}_{h}$. However, in many physical problems involving incompressible flows, the velocity is usually the physical quantity of interest. Furthermore, $P$ in the EMAC case has no physical meaning, and $c^{1}_{h}$ is naturally consistent for the case of the kinematic pressure when $\theta=0$. Therefore, the only physically meaningful expression of $P$ affected by the inconsistency is the Bernoulli function $p+\frac{1}{2}\ubold\cdot\ubold$ corresponding to $\theta=1$. This observation motivates the following new class of dG methods.

\item Scheme \textsf{dG2}

To guarantee the consistency of \eqref{compact}, we discard the two inconsistent terms in \eqref{ch_rotational} and present an alternative convective form
\begin{equation} \label{ch_rotational_2}
    \begin{aligned}
c^{2}_{h} \left(\ubold_h; \ubold_h, \vbold_h \right)&:= (\ubold_h \cdot \nabla_h \ubold_h,\vbold_h)+ \frac{1-\theta}{2} (\left(\nabla_h \cdot \ubold_h \right) \ubold_h,\vbold_h)- \iipbf{ \left( \llcurve\ubold_h\rrcurve \cdot \nbold \right) \llbracket \ubold_h \rrbracket}{\llcurve \vbold_h \rrcurve}\\
&\quad-\frac{1-\theta}{2}\iipbf{\llbracket \ubold_h \rrbracket\cdot\nbold}{\llcurve \ubold_h\cdot\vbold_h \rrcurve}+ \ipbf{\zeta \left| \llcurve\ubold_h\rrcurve \cdot \bm{n} \right|\llbracket \ubold_h \rrbracket}{\llbracket \vbold_h \rrbracket}\\
&\quad+\theta\iipbf{ \left( \llcurve\vbold_h\rrcurve \cdot \nbold \right) \llbracket \ubold_h \rrbracket}{\llcurve \ubold_h \rrcurve}-\theta(\vbold_h \cdot \nabla_h \ubold_h,\ubold_h).
\end{aligned} 
\end{equation}
The corresponding scheme \eqref{compact} with $c_h=c_h^2$ is denoted by \textsf{dG2}. Direct calculation shows that \textsf{dG2} is consistent. Coefficients of the second and fourth terms in $c_h^2$ are modified for stability analysis, see Lemma \ref{positivech} for details.

In particular, $c^2_h$ in \eqref{ch_rotational_2} for the Bernoulli function ($\theta=1$) and EMAC function ($\theta=-1$) reduces to 
\begin{align}
\nonumber c^{R}_{h} \left(\ubold_h; \ubold_h, \vbold_h \right) &= (\ubold_h \cdot \nabla_h \ubold_h,\vbold_h)- \iipbf{ \left( \llcurve\ubold_h\rrcurve \cdot \nbold \right) \llbracket \ubold_h \rrbracket}{\llcurve \vbold_h \rrcurve}\nonumber\\
\nonumber& +\iipbf{ \left( \llcurve\vbold_h\rrcurve \cdot \nbold \right) \llbracket \ubold_h \rrbracket}{\llcurve \ubold_h \rrcurve}-(\vbold_h \cdot \nabla_h \ubold_h,\ubold_h)\\
\nonumber&+\ipbf{\zeta \left| \llcurve\ubold_h\rrcurve \cdot \bm{n} \right|\llbracket \ubold_h \rrbracket}{\llbracket \vbold_h \rrbracket},
\end{align} 
and
\begin{align}
\nonumber c^{E}_{h} \left(\ubold_h; \ubold_h, \vbold_h \right) &= (\ubold_h \cdot \nabla_h \ubold_h,\vbold_h)+  (\left(\nabla_h \cdot \ubold_h \right) \ubold_h,\vbold_h)- \iipbf{ \left( \llcurve\ubold_h\rrcurve \cdot \nbold \right) \llbracket \ubold_h \rrbracket}{\llcurve \vbold_h \rrcurve}\nonumber\\
\nonumber&\quad-\iipbf{\llbracket \ubold_h \rrbracket\cdot\nbold}{\llcurve \ubold_h\cdot\vbold_h \rrcurve}+ \ipbf{\zeta \left| \llcurve\ubold_h\rrcurve \cdot \bm{n} \right|\llbracket \ubold_h \rrbracket}{\llbracket \vbold_h \rrbracket}\\
\nonumber&\quad-\iipbf{ \left( \llcurve\vbold_h\rrcurve \cdot \nbold \right) \llbracket \ubold_h \rrbracket}{\llcurve \ubold_h \rrcurve}+(\vbold_h \cdot \nabla_h \ubold_h,\ubold_h),
\end{align} 
respectively.

\item Scheme \textsf{Hdiv}

Replacing $\bm{V}_h$ with $\bm{V}^{\text{div}}_h$ in \eqref{compact} and taking $c_h=c_h^1$ or $c_h^2$ yields an $H(\text{div})$ conforming numerical scheme (denoted as \textsf{Hdiv}) for \eqref{incompressible_Euler}. It is well-known that $\bm{V}^{\text{div}}_h\times Q_h$ is an inf-sup stable pair and  $\nabla\cdot\bm{V}^{\text{div}}_h=Q_h$. Hence the incompressibility  $\nabla\cdot\bm{u}_h=0$ holds pointwise via \eqref{compact2}.  Another essential feature of $\bm{V}^{\text{div}}_h$ is that $\llbracket\bm{v}_h\cdot\bm{n}\rrbracket=0$ on $\mathcal{F}_h$. Then the penalty term $d_h(\bm{u}_h,\bm{v}_h)$ of \textsf{Hdiv} vanishes for any $\bm{v}_h\in \bm{V}^{\text{div}}_h$. 
Let $\bm{Z}_h^{\text{div}}=\{\bm{v}_h\in\bm{V}_h^{\text{div}}: \nabla\cdot\bm{v}_h=0\}$. Restricting to $\bm{V}^{\text{div}}_h\subset\bm{V}_h$, for all $\bm{u}_h, \bm{v}_h\in \bm{Z}_h^{\text{div}},$ the convective form for \textsf{Hdiv} satisfies
\begin{equation}\label{hdiv_convective}
    \begin{aligned}
&c^1_{h} \left(\ubold_h; \ubold_h, \vbold_h \right)= c^2_{h} \left(\ubold_h; \ubold_h, \vbold_h \right)\\
&=(\ubold_h \cdot \nabla_h \ubold_h,\vbold_h)- \iipbf{ \left( \ubold_h \cdot \nbold \right) \llbracket \ubold_h \rrbracket}{\llcurve \vbold_h \rrcurve}+ \iipbf{\zeta \left| \ubold_h\cdot \bm{n} \right|\llbracket \ubold_h \rrbracket}{\llbracket \vbold_h \rrbracket},
\end{aligned} 
\end{equation}
where $\llbracket\bm{v}_h\cdot\bm{n}\rrbracket=0,$ $\nabla\cdot\bm{u}_h=0$, and 
$\iipbf{ \left( \llcurve\vbold_h\rrcurve \cdot \nbold \right) \llbracket \ubold_h \rrbracket}{\llcurve \ubold_h \rrcurve}=(\vbold_h \cdot \nabla_h \ubold_h,\ubold_h)$ are used.
Therefore, \textsf{dG1} and \textsf{dG2} recover the same numerical velocity when the velocity is $H(\text{div})$-conforming. They generalize the $H(\text{div})$ conforming method in \cite{chen2020h}, where the recovery of the same numerical velocity is discussed and shown experimentally under the setting of $H(\text{div})$-conforming with consistent formulation. 

In summary, the scheme \textsf{Hdiv} is to find $(\bm{u}_h(t),P_h(t))\in\bm{V}^{\text{div}}_h\times Q_h,$
\begin{subequations}\label{compact_euler}
\begin{align}
&(\partial_t\ubold_h,\vbold_h) +{c}_h \left(\ubold_h; \ubold_h, \vbold_h \right)- \left( \nabla\cdot\vbold_h, P_h \right)=0,\quad\forall\bm{v}_h\in\bm{V}^{\text{div}}_h,\label{euler_1}\\
& \left(\nabla\cdot\ubold_h, q_h \right) = 0,\quad \forall q_h\in Q_h,\label{euler_2}
\end{align}
\end{subequations}
which is a special case of \eqref{compact}.

\end{itemize}
So far, we obtain two new DG schemes \textsf{dG1} and \textsf{dG2}, where \textsf{dG2} is exactly consistent and \textsf{dG1} is weakly consistent by penalization. In addition, \textsf{dG1}  and \textsf{dG2} reduce to the H(div) conforming scheme \textsf{Hdiv} provided $\bm{V}_h$ is replaced with $\bm{V}_h^{\text{div}}$.

\begin{lemma}[Positivity of $c^{1}_{h}$ and $c^{2}_{h}$]\label{positivech}
Assume  $\zeta_F\geq 0.5|1-\theta|$ for all $F\in\mathcal{F}_{h}^{\partial}$ when $\vbold_h\in \bm{V}_h,$ or assume $\zeta_F$ be arbitrary non-negative number when $\vbold_h\in \bm{V}_h^{\emph{\text{div}}},$ then we have
\begin{align*}
    c^{1}_{h}\left(\vbold_h; \vbold_h, \vbold_h \right)=c^{2}_{h}\left(\vbold_h; \vbold_h, \vbold_h \right)\geq 0.
\end{align*}
\end{lemma}
\begin{proof}
Direct calculation confirms the following identities
\begin{align*}
&(\bm{\beta}_h \cdot \nabla_h \vbold_h,\vbold_h)+ \frac{1}{2} (\left(\nabla_h \cdot \bm{\beta}_h \right) \vbold_h,\vbold_h)
=\frac{1}{2}  \ipbt{\vbold_h}{\vbold_h \left(\bm{\beta}_h \cdot \nbold \right)},\\
&\frac{1}{2}  \ipbt{\vbold_h}{\vbold_h \left(\bm{\beta}_h \cdot \nbold \right)}= \frac{1}{2} \iipbf{\llbracket \bm{\beta}_h \rrbracket}{\nbold \llcurve \vbold_h \cdot \vbold_h \rrcurve} + \iipbf{ \left(\llcurve \bm{\beta}_h \rrcurve \cdot \nbold \right) \llbracket \vbold_h \rrbracket}{\llcurve \vbold_h \rrcurve}+\frac{1}{2}\ipbtobf{\bm{\beta}_h\cdot\nbold}{\vbold_h \cdot \vbold_h },
\end{align*}
It follows that
\begin{align*}
    c^0_{h} \left(\bm{\beta}_h; \vbold_h, \vbold_h \right) =  \zeta \iipbf{\left| \llcurve\bm{\beta}_h\rrcurve \cdot \nbold \right|\llbracket \vbold_h \rrbracket}{\llbracket \vbold_h \rrbracket}+ \ipbtobf{\zeta|\bm{\beta}_h\cdot\nbold|+0.5\bm{\beta}_h\cdot\nbold}{\vbold_h \cdot \vbold_h }.
\end{align*}
Using  \eqref{ch_rotational}, \eqref{ch_rotational_2}, and the previous equation, we have
\begin{equation}\label{chuuu}
    \begin{aligned}
&c^{1}_{h} \left(\bm{v}_h; \bm{v}_h, \bm{v}_h \right)=c^{2}_{h} \left(\bm{v}_h; \bm{v}_h, \bm{v}_h \right) =(1-\theta)\bigg( (\bm{v}_h \cdot \nabla_h \bm{v}_h,\bm{v}_h)+ \frac{1}{2} (\left(\nabla_h \cdot \bm{v}_h \right) \bm{v}_h,\bm{v}_h)\\
&-\iipbf{ \left( \llcurve\bm{v}_h\rrcurve \cdot \nbold \right) \llbracket \bm{v}_h \rrbracket}{\llcurve \bm{v}_h \rrcurve}-\frac{1}{2}\iipbf{\llbracket \bm{v}_h \rrbracket\cdot\nbold}{\llcurve \bm{v}_h\cdot\bm{v}_h \rrcurve}+ \ipbf{\zeta \left| \llcurve\bm{v}_h\rrcurve \cdot \bm{n} \right|\llbracket \bm{v}_h \rrbracket}{\llbracket \bm{v}_h \rrbracket}\bigg)\\
&+\theta\ipbf{\zeta \left| \llcurve\bm{v}_h\rrcurve \cdot \bm{n} \right|\llbracket\bm{v}_h \rrbracket}{\llbracket\bm{v}_h \rrbracket}\\
&=\zeta \iipbf{\left| \llcurve\bm{v}_h\rrcurve \cdot\bm{v} \right|\llbracket \bm{v}_h \rrbracket}{\llbracket \bm{v}_h \rrbracket}+ \ipbtobf{\zeta|\bm{v}_h\cdot\nbold|+0.5(1-\theta)\bm{v}_h\cdot\nbold}{\bm{v}_h \cdot \bm{v}_h }.
\end{aligned} 
\end{equation}

Finally, we conclude the proof by using $\zeta_{F}\geq 0.5|1-\theta|$ for $F\in\mathcal{F}_{h}^{\partial}$ when $\vbold_h\in \bm{V}_h,$ or by using the boundary condition of $\bm{v}_h$ when it belongs to $\bm{V}_h^{\text{div}}$.
\end{proof}\qed

Now we are in a position to present the following semi-discrete stability.
\begin{theorem}[Semi-discrete Stability Estimate]\label{stability_theorem} Let the assumption in Lemma \ref{positivech} hold. Then for all $0\leq t\leq T$, Schemes \textsf{dG1}, \textsf{dG2}, and \textsf{Hdiv} satisfy
\begin{align*}
      \| \ubold_h(t) \|_{L^2 (\Omega)} \leq \left\| \ubold_h \left(0 \right) \right\|_{L^2 (\Omega)}.
\end{align*}
\end{theorem}
\begin{proof}
Taking $\bm{v}_h = \bm{u}_h$ in \eqref{compact1} and $q_h = P_h$ in \eqref{compact2}, we have
\begin{align*}
\frac{1}{2} \frac{d}{dt} \left\| \bm{u}_h \right\|_{L^2 (\Omega)}^2 + c_h \left(\ubold_h; \ubold_h, \ubold_h \right)+d_h(\ubold_h,\ubold_h)=0.  
\end{align*}
It then follows from the positivity of Lemma \ref{positivech} and $d_h(\bm{u}_h,\bm{u}_h)$ that
\begin{align*}
   \| \bm{u}_h\|_{L^2(\Omega)}\|\bm{u}_h(t)\|^\prime_{L^2(\Omega)} \leq 0,
\end{align*}
which implies
$\|\bm{u}_h(t)\|^\prime_{L^2(\Omega)} \leq 0.$
Integrating this inequality completes the proof.
\end{proof}\qed

We use the Crank-Nicolson scheme to discretize the time direction. Let the time interval $[0,T]$ be partitioned into $0=t_0<t_1<\cdots<t_{N-1}<t_N=T$. Let $\tau_n:=t_{n+1}-t_{n}$ and 
\begin{equation*}
    \delta_t\bm{u}_h^n:=\frac{\bm{u}_h^{n+1}-\bm{u}_h^n}{\tau_n},\quad\bm{u}_h^{n+\frac{1}{2}}:=\frac{\bm{u}_h^{n+1}+\bm{u}_h^n}{2}.
\end{equation*}
Let $\bm{u}_h^0\in\bm{V}_h$ be a suitable interpolation of $\bm{u}_0=\bm{u}(0)$.
The fully discrete scheme for \eqref{incompressible_Euler} is to find $\{(\bm{u}_h^{n+1},P_h^{n+\frac{1}{2}})\}_{n=0}^{N-1}\subset\bm{V}_h\times Q_h,$ such that for all $n,$
\begin{subequations}\label{CrankNicolson}
\begin{align} &\left(\delta_t\bm{u}^n_h,\vbold_h\right) + c_h\left(\bm{u}_h^{n+\frac{1}{2}};\bm{u}_h^{n+\frac{1}{2}},\bm{v}_h \right) - b_h \left( \bm{v}_h, P_h^{n+\frac{1}{2}}\right) + d_h(\ubold_h^{n+\frac{1}{2}},\vbold_h)=0,\quad\forall\bm{v}_h\in\bm{V}_h,\label{fulldis_1}\\
& b_h \left(\bm{u}^{n+\frac{1}{2}}_h, q_h \right) = 0,\quad\forall q_h\in Q_h,\label{fulldis_2}
\end{align}
\end{subequations}
where $c_h=c_h^1$ or $c_h^2.$ When $\bm{V}_h$ is replaced by $\bm{V}^{\text{div}}_h$,  \eqref{CrankNicolson} is  the fully discrete \textsf{Hdiv} scheme. In \eqref{CrankNicolson}, $\bm{u}_h^n$ and $P_h^{n+\frac{1}{2}}$ approximate $\bm{u}_h(t_n)$ and $P_h(t_n+\frac{1}{2}\tau_n)$, respectively. The next theorem confirms the stability of \eqref{CrankNicolson}.
\begin{theorem}[Fully Discrete Energy Estimate]\label{fullystability}Let the assumptions in Theorem \ref{stability_theorem} hold. We have\label{fully_stability}
\begin{align*}
   \|\ubold_h^{n+1}\|_{L^{2}(\Omega)}\leq\|\ubold_h^{n}\|_{L^{2}(\Omega)}\quad\text{ for }0\leq n\leq N-1.
\end{align*}
\end{theorem}
\begin{proof}
Let $\vbold_h=\bm{u}_h^{n+\frac{1}{2}}$ and $q_h=P_h^{n+\frac{1}{2}}$.  Add \eqref{fulldis_1} and \eqref{fulldis_2} we obtain
\begin{align} \label{conservation_equation}
\frac{1}{2\tau_n}\big(\|\ubold_h^{n+1}\|_{L^{2}(\Omega)}^{2}-\|\ubold_h^{n}\|_{L^{2}(\Omega)}^{2}\big)+c_h(\bm{u}_h^{n+\frac{1}{2}};\bm{u}_h^{n+\frac{1}{2}},\bm{u}_h^{n+\frac{1}{2}})+d_h(\ubold_h^{n+\frac{1}{2}},\ubold_h^{n+\frac{1}{2}})=0.
\end{align}
Using the positivity of $c_h$, $d_h$ and \eqref{conservation_equation}, we have 
\begin{align*} 
\|\ubold_h^{n+1}\|_{L^{2}(\Omega)}^{2}-\|\ubold_h^{n}\|_{L^{2}(\Omega)}^{2}\leq 0.
\end{align*}
The proof is complete.
\end{proof}\qed

The stability of $P_h$ is guaranteed by a discrete inf-sup condition on $b_h$, see, e.g.,  \cite{akbas2018analogue}.
\begin{remark} \label{NA}
Consider the incompressible Navier-Stokes equations 
\begin{align*}
\partial_t\ubold + \nabla \cdot \left( \ubold \otimes \ubold + P\mathbb{I} \right) - \theta\ubold\cdot\nabla\ubold^{T}-\nu\nabla \cdot \bm{\tau}(\bm{u})  = 0,  \qquad & \text{in} \quad \left(0, T\right]\times \Omega,\\
\nabla \cdot \ubold = 0, \qquad & \text{in} \quad \left(0, T\right]\times \Omega,\\  
\ubold=\bm{0},\qquad& \text{on} \quad \left(0, T\right] \times \partial\Omega,\\
\ubold(0,\xbold)=\ubold_0(\xbold),\qquad& \text{in} \quad \Omega,
\end{align*}
where $\bm{\tau}(\bm{u})=\nabla\bm{u},\text{ or }\nabla\bm{u}+ \nabla\bm{u}^T,\text{ or }\nabla\bm{u}+ (\nabla\bm{u})^T - \frac{2}{3} \left(\nabla\cdot \bm{u}\right) \mathbb{I}.$ To discretize the term $\nabla\cdot\bm{\tau}(\bm{u}),$
we introduce the following viscous bilinear form 
\begin{equation*}
a_h(\vbold_h, \wbold_h)=( \bm{\tau}_h(\bm{v}_h),\nabla_h \wbold_h)   -\ipbf{\llbracket \vbold_h \rrbracket}{\llcurve  \bm{\tau}_h(\bm{w}_h) \rrcurve \bm{n}}  
-\ipbf{\llbracket \wbold_h \rrbracket}{\llcurve  \bm{\tau}_h(\bm{v}_h) \rrcurve \bm{n}} 
+\ipbf{\eta h^{-1} \llbracket \vbold_h \rrbracket}{\llbracket \wbold_h \rrbracket},
\end{equation*}
where $\eta>0$ is a user-specified parameter, and $\bm{\tau}_h(\bm{u}_h)=\nabla_h\bm{u}_h,\text{ or } \nabla_h\bm{u}_h+ \nabla_h\bm{u}_h^T,\text{ or }\nabla_h\bm{u}_h+ (\nabla_h\bm{u}_h)^T - \frac{2}{3} \left(\nabla_h\cdot \bm{u}_h\right) \mathbb{I}.$ The bilinear form $a_h$ is standard in the context of interior penalty dG methods, see, e.g., \cite{BS2008}. Now it is straightforward to extend the scheme \eqref{CrankNicolson} to the Navier-Stokes equations: Find $(\bm{u}_h^n,P_h^n)\in\bm{V}_h\times Q_h$ such that
\begin{subequations}\label{compact_NS}
\begin{align}
&\left(\delta_t\bm{u}^n_h,\vbold_h\right) + c_h\left(\bm{u}_h^{n+\frac{1}{2}};\bm{u}_h^{n+\frac{1}{2}},\bm{v}_h \right) - b_h \left( \bm{v}_h, P_h^{n+\frac{1}{2}}\right)\nonumber\\
&\qquad+\nu a_h(\ubold_h^{n+\frac{1}{2}}, \vbold_h^{{n+\frac{1}{2}}})+ d_h(\ubold_h^{n+\frac{1}{2}},\vbold_h)=0,\quad\forall\bm{v}_h\in\bm{V}_h,\\
& b_h \left(\bm{u}^{n+\frac{1}{2}}_h, q_h \right) = 0,\quad\forall q_h\in Q_h.
\end{align}
\end{subequations}
The corresponding stability result is an direct extension of our previous analysis for the incompressible Euler equations by observing that $a_h$ is coercive w.r.t.~a mesh dependent norm, see, e.g., \cite{chen2020versatile,Ern2012} for a detailed discussion.
\end{remark}
\begin{remark}
Error estimates of $H(\emph{div})$-conforming  and dG methods for the incompressible Euler equations can be found in \cite{chen2020h,guzman2016h,natale2018variational}. For incompressible Navier-Stokes equations, interested readers are referred to \cite{layton2008introduction,schroeder2018towards}.
It is worth noticing that error estimation of $H(\emph{div})$-conforming methods in \cite{layton2008introduction} assumes the existence of an $L^{\infty}$-bounded  Stokes projection, which 
is generally hard to prove.
On the other hand, the error estimates of dG methods for Euler equations in \cite{guzman2016h} reply on a postprocessed $H(\emph{div})$-conforming velocity. Without normal continuity of the velocity and post-processing procedures, the error analysis of dG methods in this paper is expected to be even more difficult. We refer to \cite{schroeder2018towards} for interesting discussions on some open problems about error estimates for incompressible Navier-Stokes equations. 
\end{remark}
\section{Conservation of Physical Quantities}\label{conservation_property}
In this section, we discuss the conservation properties of the proposed scheme \eqref{CrankNicolson}. In particular we are interested in the conservation of discrete energy, linear and angular momentum of \textsf{dG1}, \textsf{dG2}, and \textsf{Hdiv}.

\begin{theorem}[Energy Conservation]\label{disenergy} Let the assumption in Theorem \ref{fully_stability} holds. Scheme \eqref{CrankNicolson} with $c_h=c_h^1$ or $c_h=c_h^2$ conserves the total energy if and only if  $\theta=1$, $\gamma=0$ and $\zeta=0$ on $\mathcal{F}_h$. Scheme \eqref{CrankNicolson} with $\bm{V}_h=\bm{V}^{\emph{\text{div}}}_h$ conserves the total energy if and only if $\zeta=0$ on $\mathcal{F}_h$.  \label{energy_conservation_DG}
\end{theorem}
\begin{proof}
From Lemma \ref{positivech}, Eqs.~\eqref{d_h} and \eqref{conservation_equation}, we observe that the only possibility that the total energy is conserved is when both $c_h(\bm{u}_h^{n+\frac{1}{2}};\bm{u}_h^{n+\frac{1}{2}},\bm{u}_h^{n+\frac{1}{2}})$ and $d_h(\ubold_h^{n+\frac{1}{2}},\ubold_h^{n+\frac{1}{2}})$ vanish. Clearly $d_h(\ubold_h^{n+\frac{1}{2}},\ubold_h^{n+\frac{1}{2}})=0$ implies $\gamma=0$. In view of \eqref{chuuu}, $c_h(\bm{u}_h^{n+\frac{1}{2}};\bm{u}_h^{n+\frac{1}{2}},\bm{u}_h^{n+\frac{1}{2}})=0$ implies $\zeta=0.$ Finally, stability condition $\zeta_{F}\geq 0.5|1-\theta|$ for $F\in\mathcal{F}_h^\partial$ implies that $\theta$ can only be 1. 

For the H(div) scheme, recall that $d_h(\bm{u}_h,\bm{u}_h)=0$. Then energy conservation is equivalent to
\begin{align*}
\nonumber &c_{h} \left(\ubold_h; \ubold_h, \ubold_h \right)=\iipbf{\zeta \left| \ubold_h\cdot \bm{n} \right|\llbracket \ubold_h \rrbracket}{\llbracket \ubold_h \rrbracket}=0.
\end{align*}
The proof is complete.
\end{proof}\qed

We use $\gamma$ to weakly enforce the consistency of Scheme \textsf{dG1}. Hence the requirement $\gamma=0$ does not make sense in practice for \textsf{dG1}. For \textsf{dG2}, $\gamma=0$ indicates a lack of pressure robustness. Therefore we do not believe the conservation of energy (at the sacrifice of pressure robustness) is worth pursuing in the dG formulations within the current framework, actually numerical experiments show that both the Newton nonlinear solver with the MUMPS linear solver or with the GMRES linear solver break down after a few time steps. 

We then consider the preservation of discrete linear and angular momentum. Straightforward calculations show that dG schemes do not conserve the linear and angular momentum. For the H(div) scheme, the test function $\bm{v}_h=\bm{e}_i$ (for linear momentum) or $\vbold_h=\bm{x}\times\bm{e}_i$ (for angular momentum) is not contained in $\bm{V}^{\text{div}}_h$, where $\bm{e}_i$ is the $i$-th unit vector. Following \cite{charnyi2017conservation}, we assume that $\bm{u}_h= 0$ and $P_h= 0$ in an inner neighborhood $\Omega^b$ of the boundary $\partial\Omega$ and take $\Omega^i=\Omega\backslash\Omega^b.$ Given a continuous function $\bm{\phi}$, we take $\bm{\phi}^i=\bm{\phi}$ on $\Omega^i$ and $\bm{\phi}^i|_{\Omega^b}$ to be suitably defined to meet the boundary condition.
Integration by parts yields
\begin{align*}
\iipbf{ \left( \ubold_h \cdot \nbold \right) \llbracket \ubold_h \rrbracket}{\bm{\phi}^i}&=(\ubold_h\cdot\nabla_h\ubold_h,\bm{\phi}^i)+(\ubold_h,\nabla_h\cdot(\bm{\phi}^i\otimes\ubold_h))  \\[1.5ex]
&=(\ubold_h\cdot\nabla_h\ubold_h,\bm{\phi}^i)+(\nabla\cdot\ubold_h,\ubold_h\cdot\bm{\phi}^i)+(\ubold_h,\ubold_h\cdot\nabla_h\bm{\phi}^i)\\[1.5ex]
&=(\ubold_h\cdot\nabla_h\ubold_h,\bm{\phi}^i)+(\ubold_h,\ubold_h\cdot\nabla\bm{\phi}^i)
\end{align*}
where we have used $\nabla\cdot\ubold_h=0$ pointwise, $\ubold_h|_{\Omega^b}=0$ and $\bm{\phi}^{i}$ is continous within $\Omega^{i}$. Using $\ubold_h|_{\Omega^b}=0$ again, we have
\begin{equation}\label{momentumeqn}
\iipbf{ \left( \ubold_h \cdot \nbold \right) \llbracket \ubold_h \rrbracket}{\bm{\phi}}=(\ubold_h\cdot\nabla_h\ubold_h,\bm{\phi})+(\ubold_h,\ubold_h\cdot\nabla\bm{\phi}).
\end{equation}
Taking $\bm{\phi}=\bm{e}_i$ in \eqref{momentumeqn}, it follows that
\begin{align}\label{linear}
\iipbf{ \left( \ubold_h \cdot \nbold \right) \llbracket \ubold_h \rrbracket}{\bm{e}_i}=(\ubold_h\cdot\nabla_h\ubold_h,\bm{e}_i).
\end{align}
For $\bm{\phi}=\bm{x}\times\bm{e}_i$ and any function $\gbold$, elementary tensor calculation shows that  $\gbold\cdot(\gbold\cdot\nabla\bm{\phi})=0$. Therefore \eqref{momentumeqn} with $\bm{\phi}=\bm{x}\times\bm{e}_i$ implies 
\begin{align}\label{angular}
\iipbf{ \left( \ubold_h \cdot \nbold \right) \llbracket \ubold_h \rrbracket}{\bm{x}\times\bm{e}_i}=(\ubold_h\cdot\nabla_h\ubold_h,\bm{x}\times\bm{e}_i).
\end{align}
As a consequence, we obtain the next theorem.
\begin{theorem}[Linear and Angular Momentum Conservation-$H$(div)] Let the assumptions in Theorem \ref{disenergy} hold. Assume that $\bm{u}_h=\bm{0}$, $P_h=0$ in $\Omega^b.$ Then Scheme \eqref{CrankNicolson} with $\bm{V}_h=\bm{V}_h^{\emph{div}}$ (in addition $k\geq 1$ if $RT$ element is used) and $c_h=c_h^1$ or $c_h^{2}$ conserve the linear and angular momentum. 
\label{linear_angular}
\end{theorem}
\begin{proof}
Let $\bm{\phi}_h\in \bm{V}_h^{\text{div}}$ be $\bm{e}_i$ or $\xbold\times\bm{e}_i$ within $\Omega^i$ and  arbitrarily defined in $\Omega^b$ to meet the boundary condition. Clearly we have 
\begin{align}\label{observation}
\nabla\cdot\bm{e}_i=0, \qquad \nabla\cdot(\bm{\xbold\times\bm{e}_i})=0,
\end{align}
and thus $\nabla\cdot\bm{\phi}_h=0$ in $\Omega^i.$
Then using $d_h=0$ and the assumption $\bm{u}_h|_{\Omega^b}=0$ $P_h|_{\Omega^b}=0$, \eqref{fulldis_1} reduces to
\begin{align*}
    &\left(\delta_t\bm{u}^n_h,\bm{\phi}_h\right) + c_h\left(\bm{u}_h^{n+\frac{1}{2}};\bm{u}_h^{n+\frac{1}{2}},\bm{\phi}_h \right)=0.
\end{align*}
The preservation of linear ($\bm{\phi}_h=\bm{e}_i$) and angular momentum ($\bm{\phi}_h=\bm{x}\times\bm{e}_i$) is equivalent to $c_h\left(\bm{u}_h^{n+\frac{1}{2}};\bm{u}_h^{n+\frac{1}{2}},\bm{\phi}_h \right)=0$. For simplicity of presentation, we omit the superscript of $\bm{u}_h$ in the rest of the proof.

It is noted  that $\bm{\phi}_h$ may not belong to $\bm{Z}_h^{\text{div}}$. Therefore one could not directly make use of \eqref{hdiv_convective}, \eqref{linear} and \eqref{angular} to conclude
\begin{align*}
    c_h(\ubold_h,\ubold_h,\bm{\phi}_h)=\iipbf{\zeta \left| \ubold_h\cdot \bm{n} \right|\llbracket \ubold_h \rrbracket}{\llbracket \bm{\phi}_h \rrbracket}.
\end{align*}
Instead one need to study $c_h^{1}$ and $c_h^{2}$ separately. We recall that \eqref{hdiv_convective} implies all $H(\text{div})$-conforming scheme recover the same numerical velocity, and the linear and angular moment only depend on the velocity field. Hence it is enough to analyze $c_{h}^1$ and $c_{h}^2$. 
\begin{itemize}
\item The case of $c_h^{1}$  

It follows from the identity
\begin{align}\label{identity_hdiv}
\iipbf{ \left( \vbold_h \cdot \nbold \right) \llbracket \ubold_h \rrbracket}{\llcurve \ubold_h \rrcurve}= (\vbold_h \cdot \nabla_h \ubold_h,\ubold_h)+\frac{1}{2} (\left(\nabla_h \cdot \vbold_h \right) \ubold_h,\ubold_h),\quad\forall \bm{u}_h,\bm{v}_h\in \bm{V}_{h}^{\text{div}}
\end{align}
and $\nabla\cdot\ubold_h=0$, $\llbracket\bm{u}_h\cdot\bm{n}\rrbracket=\llbracket\bm{v}_h\cdot\bm{n}\rrbracket=0$ for all $F\in \mathcal{F}_h$, and \eqref{ch_rotational} that
\begin{align*}
&c^1_{h} \left(\ubold_h; \ubold_h, \bm{\phi}_h \right)= (\ubold_h \cdot \nabla_h \ubold_h,\bm{\phi}_h)- \iipbf{ \left( \ubold_h \cdot \nbold \right) \llbracket \ubold_h \rrbracket}{\llcurve \bm{\phi}_h \rrcurve}+\iipbf{\zeta \left| \ubold_h\cdot \bm{n} \right|\llbracket \ubold_h \rrbracket}{\llbracket \bm{\phi}_h \rrbracket}.
\end{align*} 
We conclude
$c_h^{1}(\ubold_h,\ubold_h,\bm{\phi}_h)=0$ from \eqref{linear}, \eqref{angular}, $\llbracket\bm{\phi}_h\rrbracket=0$ within $\Omega^{i}$ and $\bm{u}_h=0$ within $\Omega^{b}$. 
\item The case of $c_h^{2}$ 

As in the case of $c_h^{1}$, we obtain from \eqref{ch_rotational_2} that  
\begin{align*}
c^2_{h} \left(\ubold_h; \ubold_h, \bm{\phi}_h \right)&= (\ubold_h \cdot \nabla_h \ubold_h,\bm{\phi}_h)- \iipbf{ \left( \ubold_h \cdot \nbold \right) \llbracket \ubold_h \rrbracket}{\llcurve \bm{\phi}_h \rrcurve}+ \iipbf{\zeta \left| \llcurve\ubold_h\rrcurve \cdot \bm{n} \right|\llbracket \ubold_h \rrbracket}{\llbracket \bm{\phi}_h \rrbracket}\\
&+\frac{\theta}{2} (\left(\nabla_h \cdot \bm{\phi}_h \right) \ubold_h,\ubold_h).
\end{align*} 
again with \eqref{observation} and thanks to the assumption that $\ubold_h=0$ within $\Omega^{b}$, we conclude that 
\begin{align*}
\frac{\theta}{2} (\left(\nabla \cdot \bm{\phi}_h \right) \ubold_h,\ubold_h)=0,
\end{align*} 
the rest of the steps are identical to those in the case of $c_h^{1}$. \qed
\end{itemize}
\end{proof}
 
\begin{remark}
The upwinding H(\emph{div}) scheme ($\zeta>0$) fails to preserve the total energy. However it is interesting to see that the upwinding H(\emph{div}) scheme still conserves the linear- and angular momentum.
\end{remark}
\section{Numerical Experiments}\label{secNE}
In this section, we present numerical simulations to test the conservation properties of the $\textsf{Hdiv}$ schemes \eqref{CrankNicolson}, and the performance of the dG schemes \eqref{CrankNicolson} and \eqref{compact_NS} in steady and unsteady problems with $c_h$ being either $c^{1}_{h}$ (\textsf{dG1}) or $c^{2}_{h}$ (\textsf{dG2}), and $\bm{\tau}_h(\bm{u}_h)=\nabla_h\ubold_h$. The viscous penalization parameter $\eta$ is taken as $\eta=3(k+1)(k+2)$ while $\gamma$ and $\zeta$ are specified in each example. 
In each experiment, we use the non-homogeneous Dirichlet boundary condition and  weakly enforce it via modifying the right hand side of \eqref{CrankNicolson} or \eqref{compact_NS}.
Our numerical simulations are performed in FEniCS \cite{LangtangenLogg2017}, an open-source computing platform for solving partial differential equations using finite element methods. For all numerical experiments in this section, we use the Newton nonlinear solver with the \textsf{MUMPS} linear solver inside FEniCS to solve the nonlinear systems of equations arising from fully discrete schemes. The absolute and relative error tolerances in the Newton solver are set to be $10^{-8}$ for dynamic problems and $10^{-10}$ for stationary problems.
\subsection{Conservation Test}\label{conservation_test}
In this subsection, we use the problem of Gresho vortex \cite{charnyi2017conservation} to test the conservation properties of the $\textsf{Hdiv}$ scheme \eqref{CrankNicolson} with both $\zeta=0$ (central flux) and $\zeta=0.5$ (upwind flux) using BDM element for $\bm{V}_h^{\text{div}}$. The space domain is $\Omega:=[-0.5,0.5]^2$ with the initial velocity and kinematic pressure fields described by 
\begin{align*}
\ubold_{\phi}(r,\phi) &=\begin{cases}
  5r,  &\text{$0\leq r\leq 0.2$},  \\   
  2-5r, &\text{$0.2\leq r\leq 0.4$},  \\ 
  0, &\text{$0.4\leq r$}, 
\end{cases} \\
\ubold_{r}(r,\phi) &=0.\\
p &=\begin{cases}
  12.5r^{2}+C_1,  &\text{$0\leq r\leq 0.2$},  \\  
  12.5r^{2}-20r+4\text{log}(r)+C_2, &\text{$0.2\leq r\leq 0.4$},  \\
  0, &\text{$0.4\leq r$}, 
\end{cases} 
\end{align*}
Here polar coordinates $\left(r, \phi \right)$ are used, and $C_1=C_2-4+4\text{log}(0.2)$ and $C_1=5.2-4\text{log}(0.4)$. Note that the initial state is an exact solution of the steady incompressible Euler equations, therefore an accurate scheme should preserve the solution when there is no forcing terms. For this problem, we run the simulation with $k=1$ and $h_{\text{max}}=0.0283$ for $10s$ with time step $\tau=0.01s$. From Figures~\ref{fig:Conservation_1} and \ref{Gresho_picture}, we have the following interesting observations
\begin{itemize}
\item Both upwind and central fluxes seem to be able to conserve the total linear momentum in both $x_{1}$ and $x_{2}$ directions.
\item Energy dissipation appears when upwinding since $\zeta>0$, but the maximum relative error of total energy is less than $0.1\%$, which is almost negligible.
\item Compared to upwind flux, the central flux fails to conserve the total angular momentum after about $5s$. This is due to the violation of the assumption $\bm{u}_h=0$ near $\partial{\Omega}$, which is also confirmed by the left column of Figure \ref{Gresho_picture}. However, one thing worth noticing is that the maximum relative error of angular momentum for central flux is within $5\%$,  acceptable for most engineering applications. Therefore one may consider the central flux being able to conserve the angular momentum. 
\item Central fluxes have a much larger (more than $25$ times) $L^{2}$ error in the velocity compared with upwind flux. Figure~\ref{Gresho_picture} indicates that central flux fails to preserve the right physics after about $3s$ and makes almost no physical sense when $t=9s$, although all global conservation properties hold up to $10s$ (if one's tolerance of relative error is within $5\%$).
\end{itemize}
Through this example, we show that the conservation of global physical quantities such as energy, linear and angular momentum may not be a good indicator for a well-behaved numerical scheme as they lack the control of local behaviors, which could be as bad as the left figure in the lower row of Figure~\ref{Gresho_picture}. In addition, we have shown that the upwinding H(div) scheme could numerically preserve linear and angular momentum, as predicted by our theoretical analysis, and this point has not been seriously discussed in the literature.

\begin{figure}[ht]
\centering
\begin{subfigure}{0.4\textwidth}
  \centering
  \includegraphics[width=1.0\linewidth,trim={0cm 0cm 0cm 0cm},clip]{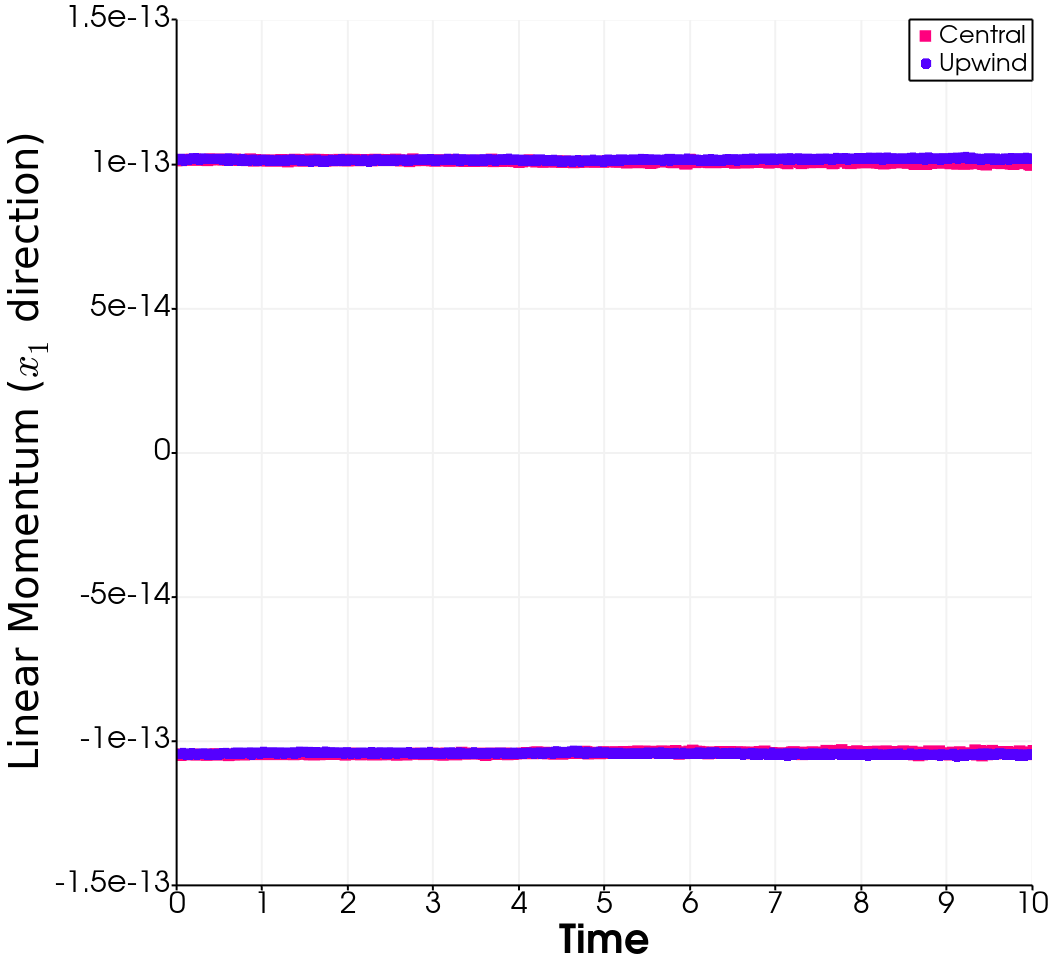}
\end{subfigure}%
\begin{subfigure}{0.4\textwidth}
  \centering
  \includegraphics[width=1.0\linewidth,trim={0cm 0cm 0cm 0cm},clip]{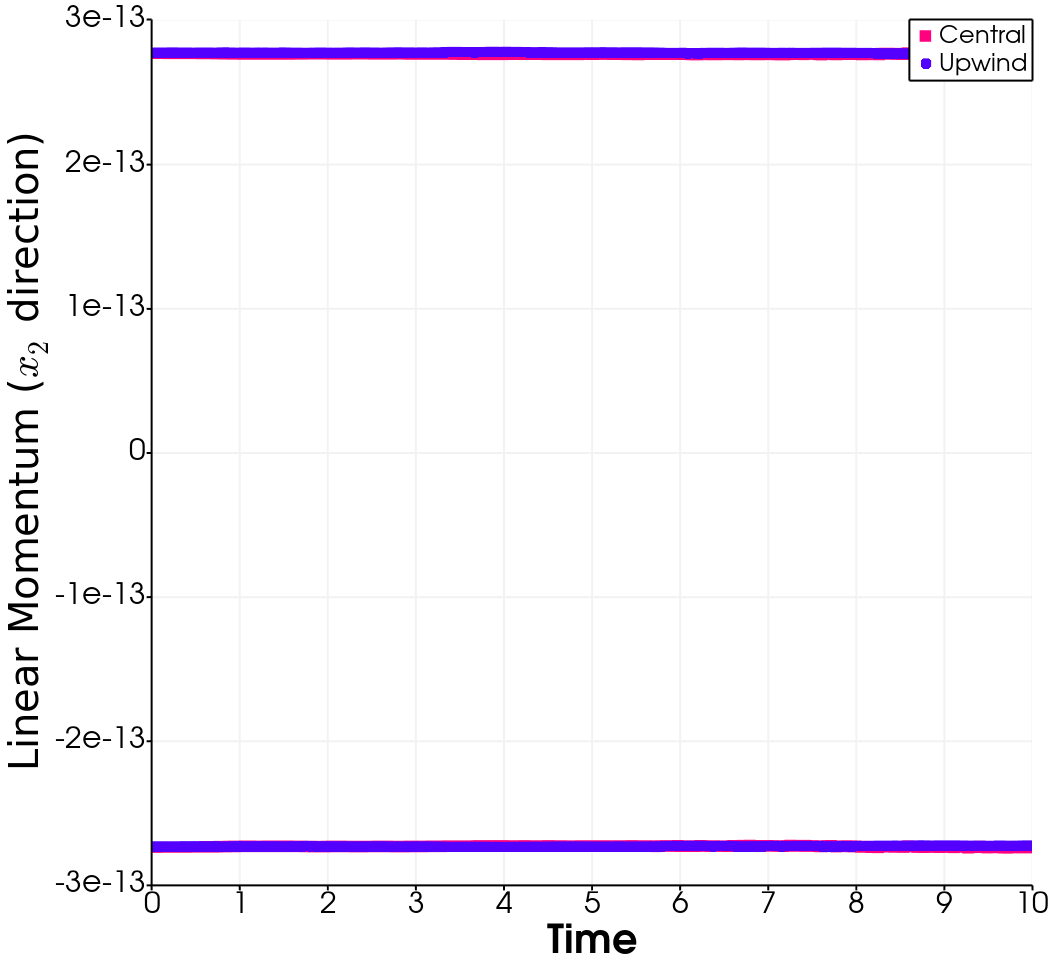}
\end{subfigure}
\caption{Gresho Vortex: Plots of time versus linear momentum in direction $x_{1}$ and $x_{2}$.}
\label{fig:Conservation_1}
\end{figure}
\begin{figure}[ht]
\centering
\begin{subfigure}{0.4\textwidth}
  \centering
  \includegraphics[width=1.0\linewidth,trim={0cm 0cm 0cm 0cm},clip]{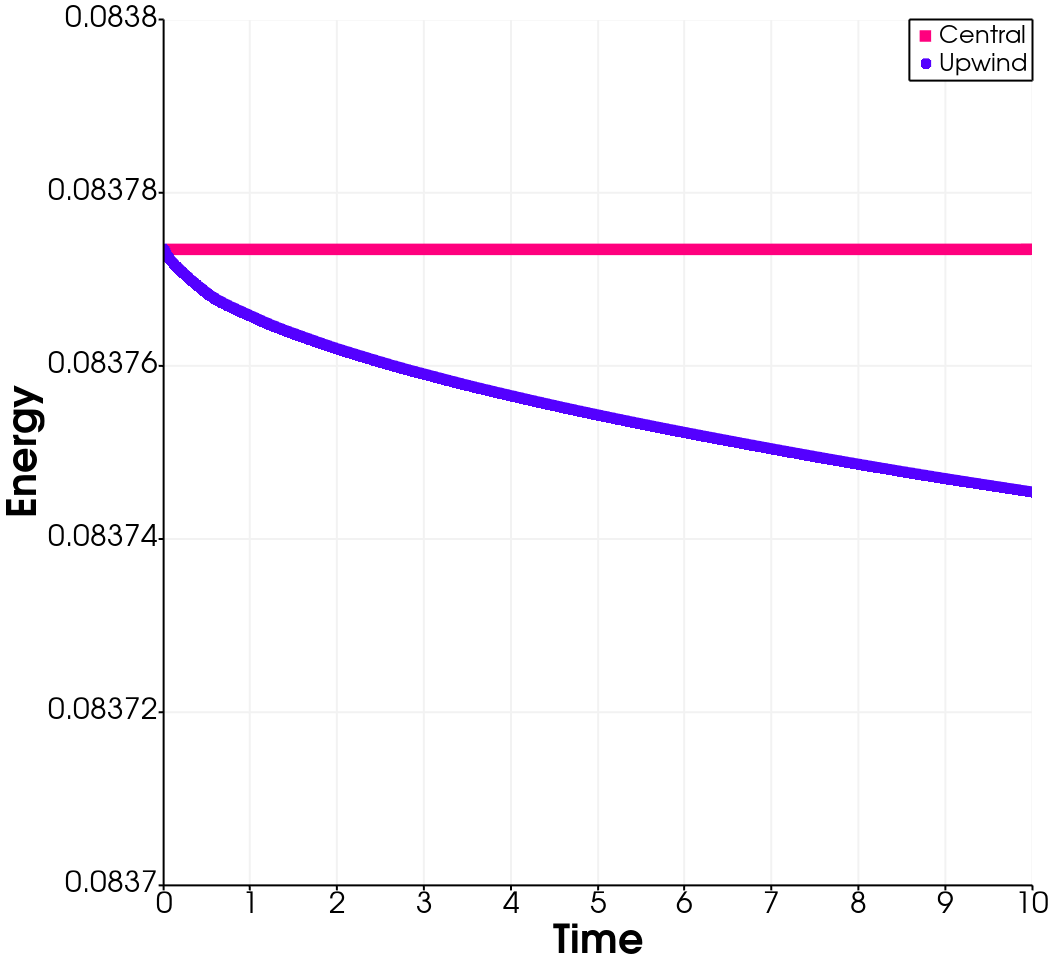}
\end{subfigure}%
\begin{subfigure}{0.4\textwidth}
  \centering
  \includegraphics[width=1.0\linewidth,trim={0cm 0cm 0cm 0cm},clip]{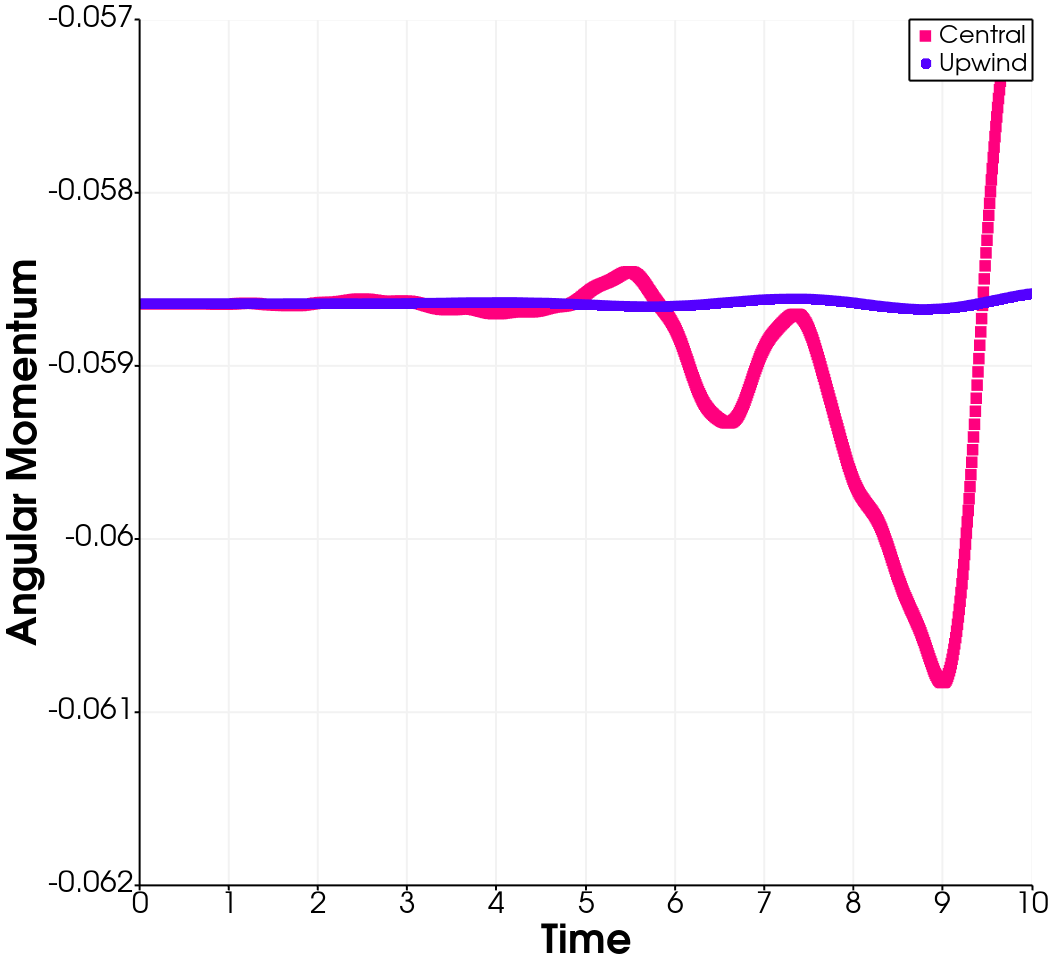}
\end{subfigure}
\begin{subfigure}{0.4\textwidth}
  \centering
  \includegraphics[width=1.0\linewidth,trim={0cm 0cm 0cm 0cm},clip]{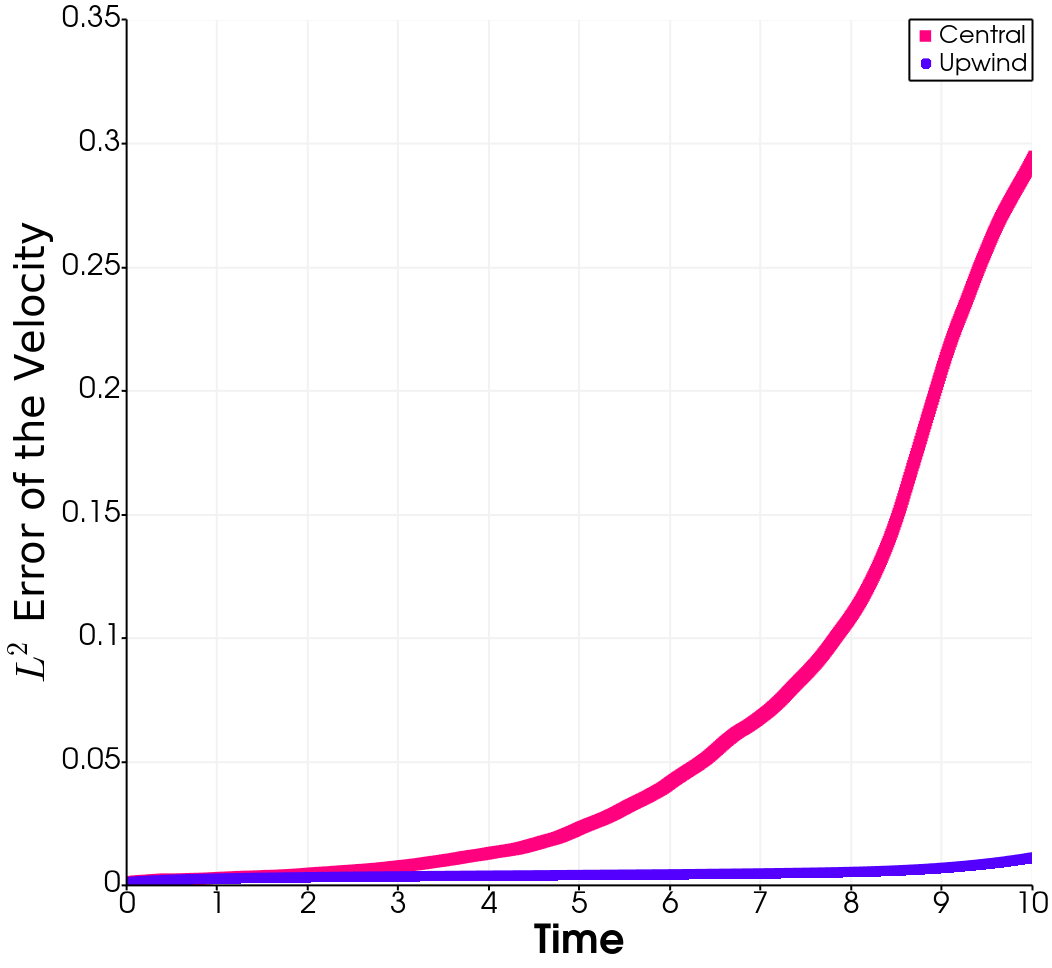}
\end{subfigure}
\caption{Gresho Vortex: Plots of time versus energy, angular momentum, and the velocity $L^{2}$ error.}
\label{fig:Conservation_2}
\end{figure}
\begin{figure*}
\centering
\begin{multicols}{2}
    \includegraphics[width=0.75\linewidth,trim={7cm 2cm 0cm 2cm},clip]{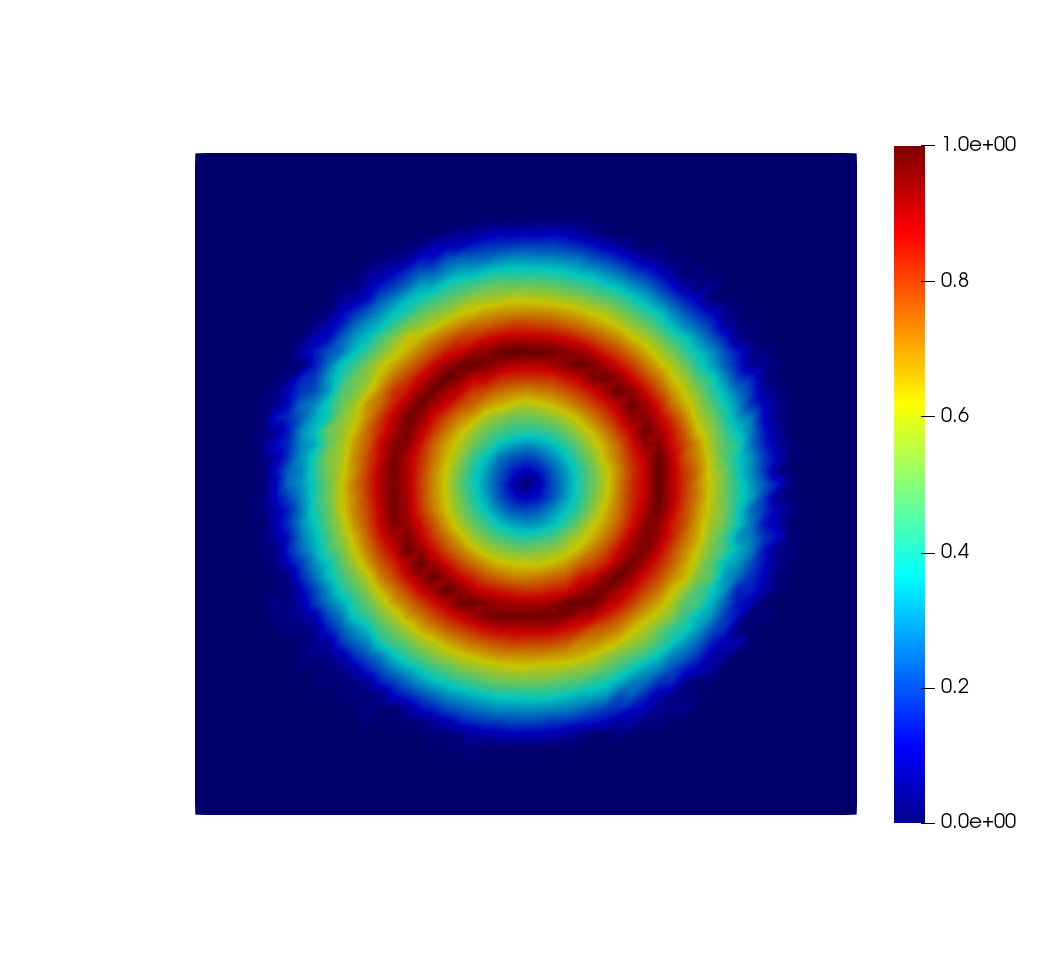}\par
    \includegraphics[width=0.75\linewidth,trim={7cm 2cm 0cm 2cm},clip]{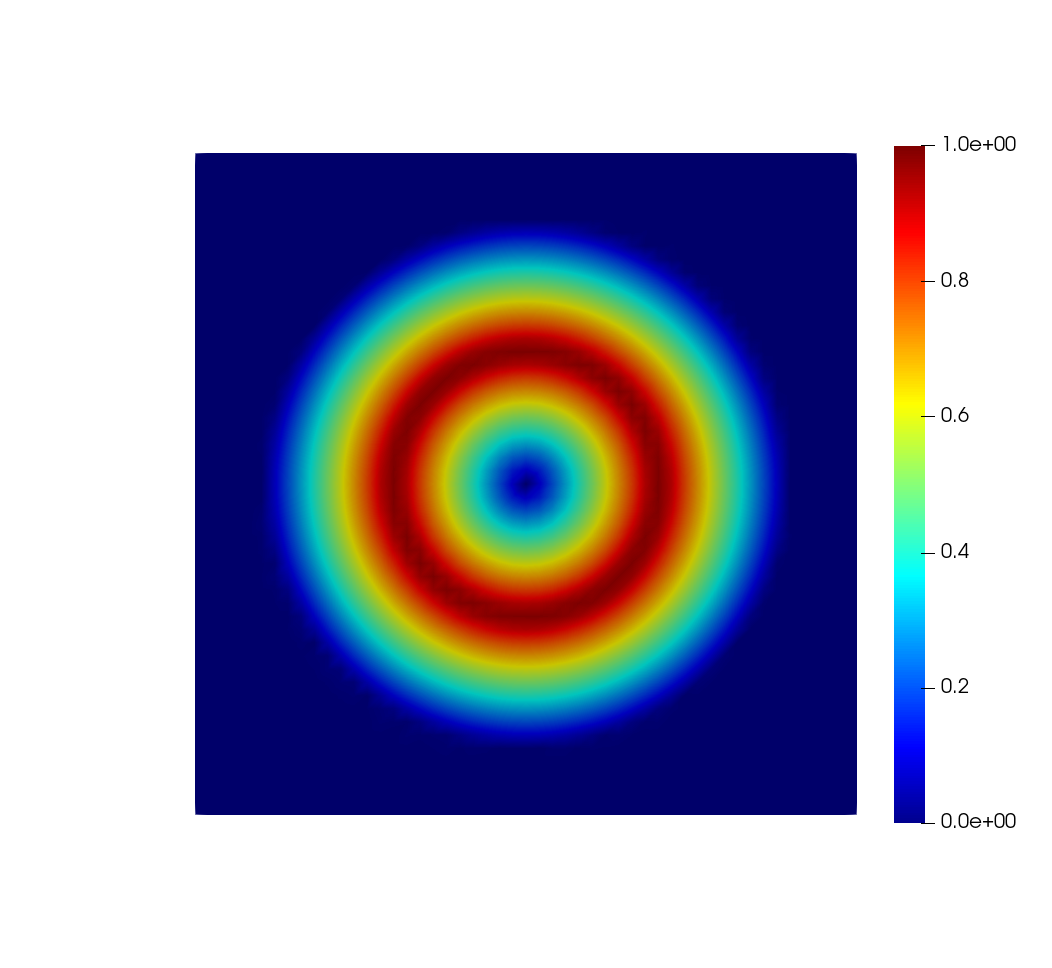}\par 
    \end{multicols}
\begin{multicols}{2}
    \includegraphics[width=0.75\linewidth,trim={7cm 2cm 0cm 2cm},clip]{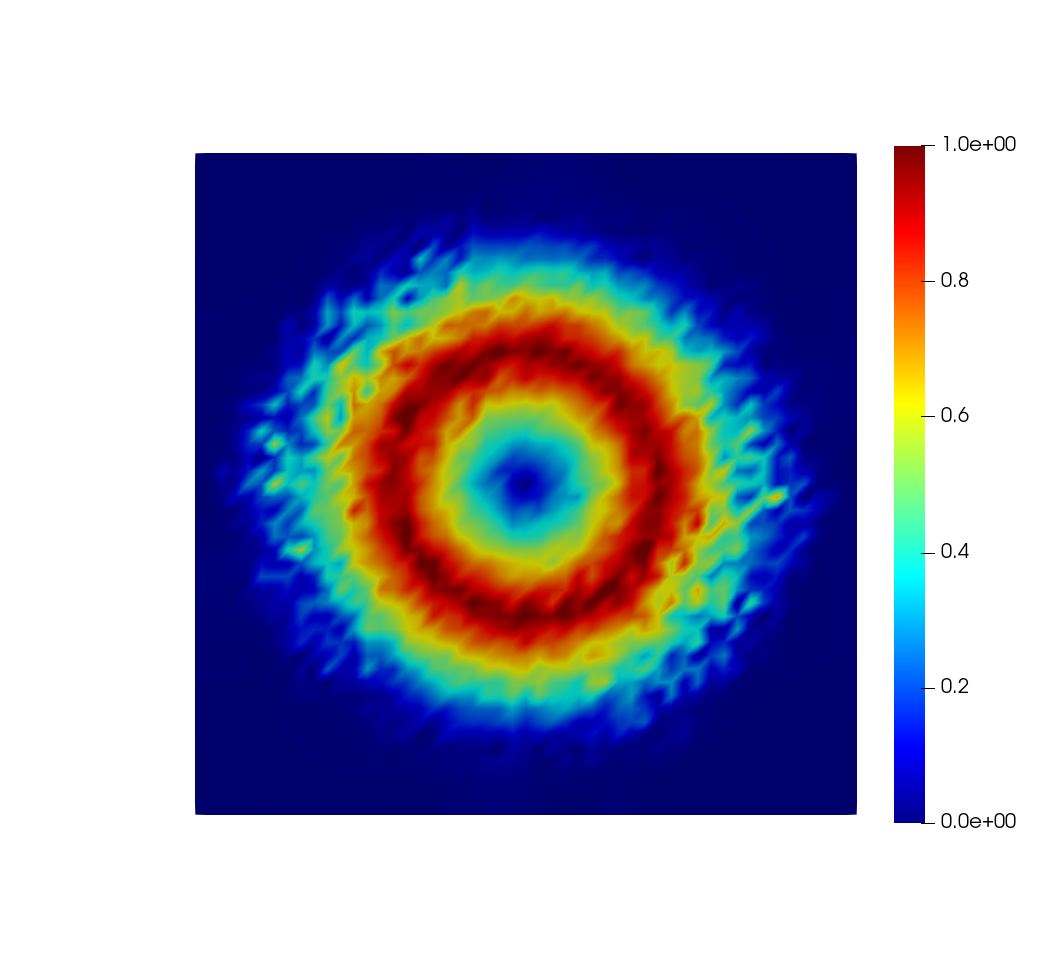}\par
    \includegraphics[width=0.75\linewidth,trim={7cm 2cm 0cm 2cm},clip]{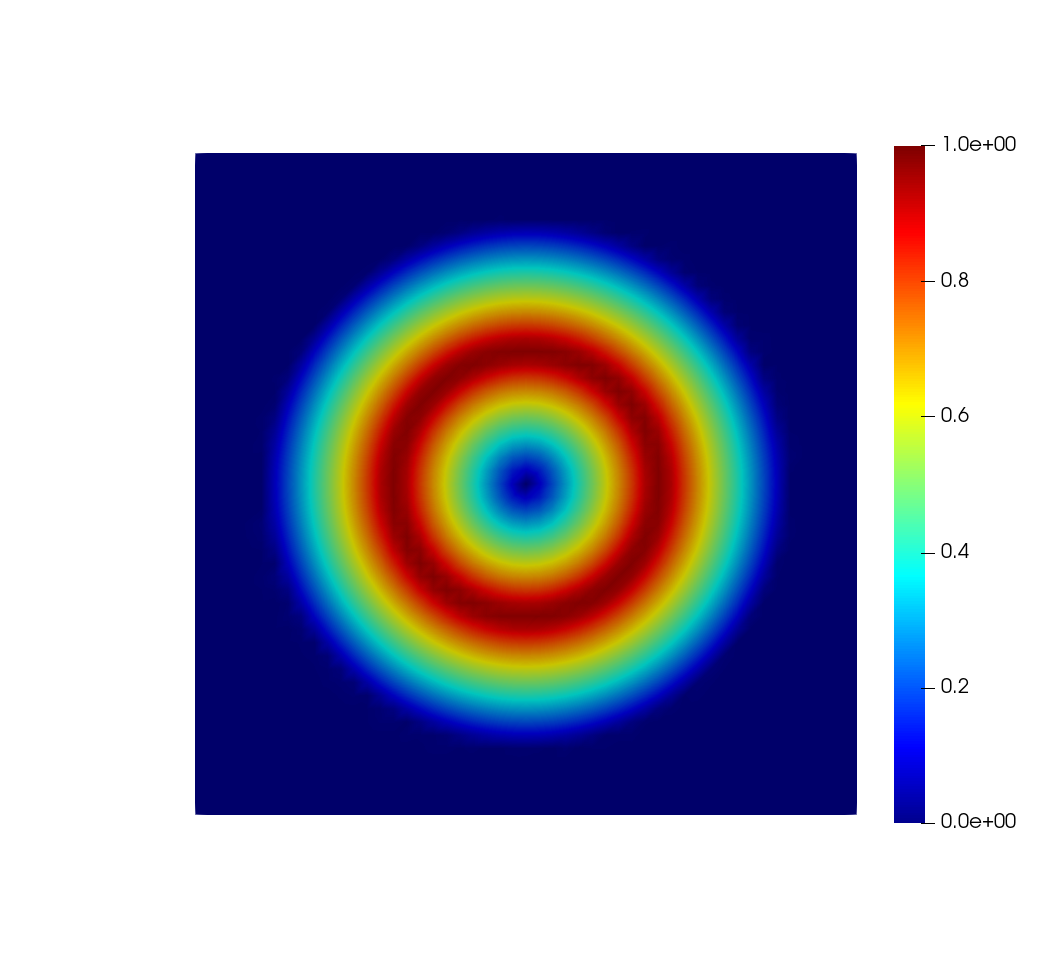}\par
\end{multicols}
\begin{multicols}{2}
    \includegraphics[width=0.75\linewidth,trim={7cm 2cm 0cm 2cm},clip]{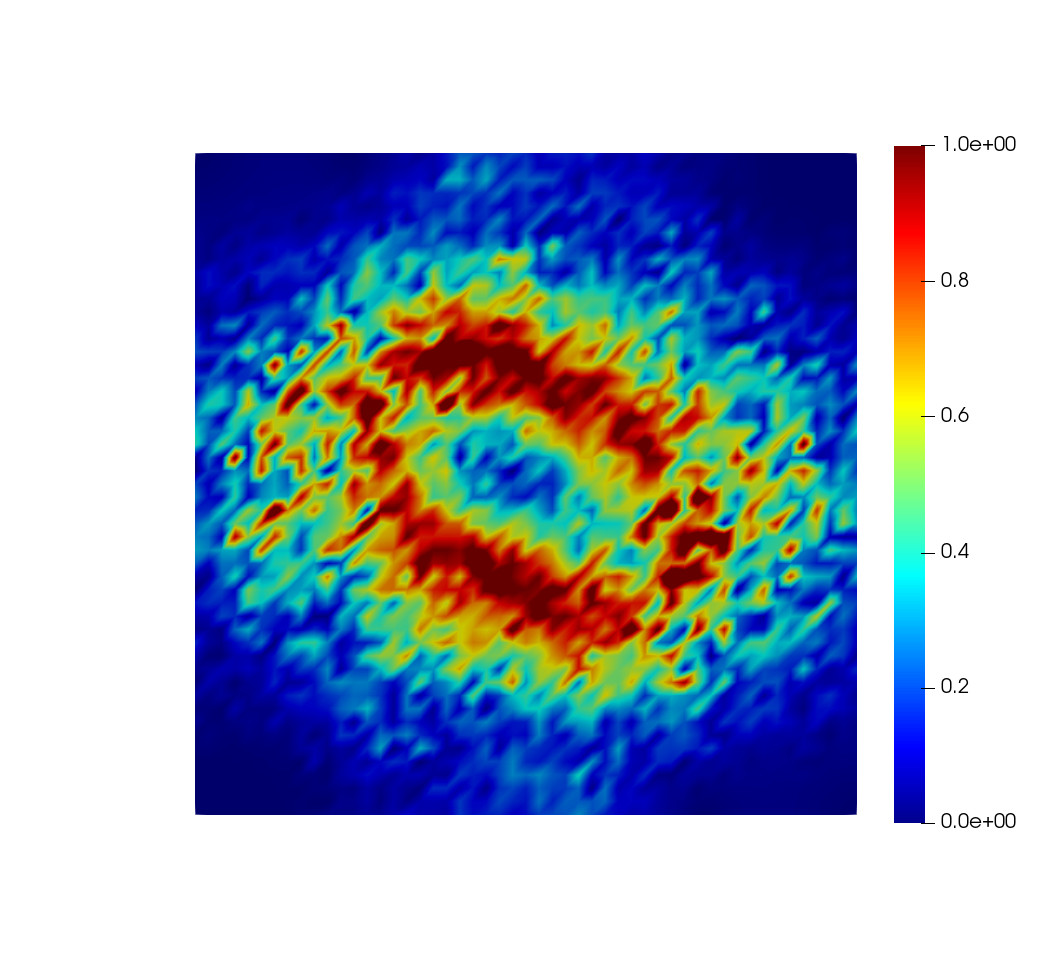}\par
    \includegraphics[width=0.75\linewidth,trim={7cm 2cm 0cm 2cm},clip]{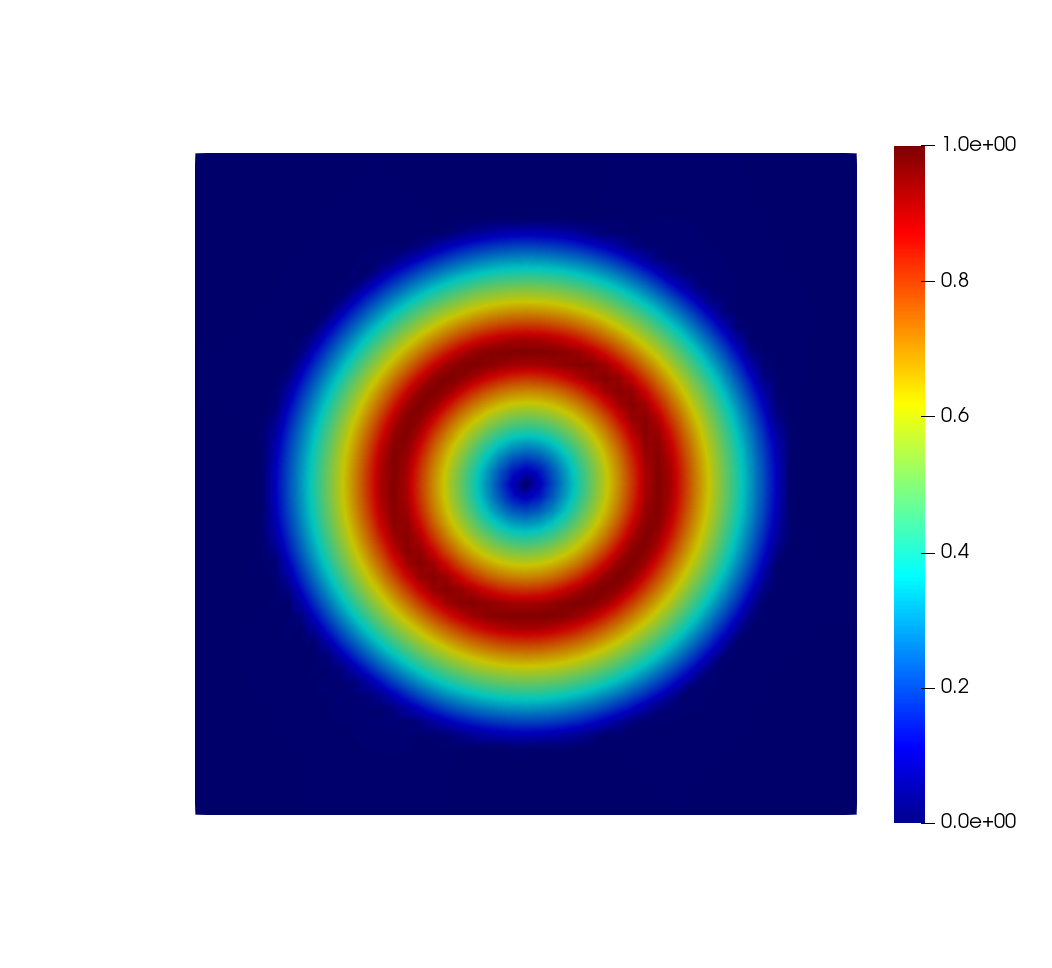}\par
\end{multicols}
\caption{Gresho vortex: Contours of the velocity magnitude with $t=3.0s$ (upper row), $t=6.0s$ (middle row) and $t=9.0s$ (lower row) when central (left) and upwind flux (right) with $h_{\text{max}}=0.0283$ and $k=1$}
\label{Gresho_picture}
\end{figure*}
\clearpage
\subsection{Accuracy of Scheme \textsf{dG2}}\label{c_2h_test}
In this subsection, we test the behavior of Scheme \textsf{dG2} and compare it with the classical scheme from \cite{akbas2018analogue,Ern2012} using the Taylor-Green vortex and the Kovasznay flow problems, which are popular benchmark problems in computational fluid dynamics  \cite{guzman2016h,chorin1968numerical,liu2000high,Ern2012}. More precisely,  Taylor-Green vortex is used with the incompressible Euler equations with the viscous terms viewed as a forcing term, while Kovasznay flow is used as an exact solution to the incompressible Navier-Stokes equations. Throughout the rest of the paper, we set $\zeta=1.0$. 

We first study the Taylor-Green vortex 
in $\mathbb{R}^2$, with the analytical solution given by \cite{guzman2016h}
\begin{align*}
    &\ubold(t,\xbold)=\bigg(\sin(x_{1})\cos(x_{2})e^{-2\nu t},-\cos(x_{1})\sin(x_{2})e^{-2\nu t}\bigg), \\ &p(t,\xbold)=\frac{1}{4}\bigg(\cos(2x_{1})+\cos(2x_{2})\bigg)e^{-4\nu t}.
\end{align*}
The space domain is $\Omega:=[0,2\pi]^2$, which is meshed with squares and each square was further split into two congruent triangles. We set $\nu=0.01$, $\gamma=1000$ and run the simulation for $0.5s$ using the time step $\tau=0.01s$ on meshes with mesh sizes $h_{\text{max}}\in\{0.8886, 0.4443, 0.2221, 0.1777\}$ and polynomial degree $k\in\{0,1,2\}$. Finally, we record the $\Tilde{c}_h$ formulation in \cite{Ern2012} below, which is 
\begin{align*}
\Tilde{c}_h \left(\ubold_h; \ubold_h, \vbold_h \right) = (\ubold_h \cdot \nabla_h \ubold_h,\vbold_h)- \iipbf{ \left( \llcurve\ubold_h\rrcurve \cdot \nbold \right) \llbracket \ubold_h \rrbracket}{\llcurve \vbold_h \rrcurve}+ \iipbf{\zeta\left| \llcurve\ubold_h\rrcurve \cdot \bm{n} \right|\llbracket \ubold_h \rrbracket}{\llbracket \vbold_h \rrbracket}.
\end{align*}
Note that the scheme in \cite{Ern2012} is designed only for steady flows with $\Tilde{c}_h$ not necessarily being positive semi-definite. Therefore this scheme is not proved to be energy stable for unsteady flow. 
\begin{table}[]
\begin{tabular}{|l|l|l|l|l|l|l|l|l|l|l|}
\hline
\multicolumn{1}{|c|}{\multirow{2}{*}{$k$}} & \multicolumn{1}{c|}{\multirow{2}{*}{$h_{\text{max}}$}} & \multicolumn{1}{c|}{\multirow{2}{*}{d.o.f}} & \multicolumn{6}{c|}{Our Schemes $\|\ubold-\ubold_{h}\|_{L^{2}(\Omega)}$}                                                                                                                                                                                                                                                    & \multicolumn{2}{c|}{\multirow{2}{*}{\begin{tabular}[c]{@{}c@{}}Scheme in \cite{akbas2018analogue,Ern2012}\\ $\|\ubold-\ubold_{h}\|_{L^{2}(\Omega)}$\\ Error\qquad   order\end{tabular}}} \\ \cline{4-9}
\multicolumn{1}{|c|}{}                   & \multicolumn{1}{c|}{}                   & \multicolumn{1}{c|}{}                       & \multicolumn{2}{c|}{\begin{tabular}[c]{@{}c@{}}$\theta=0$\\ Error\qquad order\end{tabular}} & \multicolumn{2}{c|}{\begin{tabular}[c]{@{}c@{}}$\theta=1$\\ Error\qquad   order\end{tabular}} & \multicolumn{2}{c|}{\begin{tabular}[c]{@{}c@{}}$\theta=-1$\\ Error\qquad   order\end{tabular}} & \multicolumn{2}{c|}{}                                                                                                         \\ \hline
\multirow{4}{*}{0}                       & 0.8886                                  &    1401                                         & \multicolumn{2}{l|}{2.02e-1\qquad ---}                                                                        & \multicolumn{2}{l|}{2.02e-1\qquad ---}                                                                      & \multicolumn{2}{l|}{2.02e-1\qquad ---}                                                                  & \multicolumn{2}{l|}{2.01e-1\qquad ---}                                                                                                         \\ \cline{2-11} 
& 0.4443                                  &      5601                                       & \multicolumn{2}{l|}{4.44e-2\qquad 2.18}                                                                        & \multicolumn{2}{l|}{4.44e-2\qquad 2.18}                                                                      & \multicolumn{2}{l|}{4.45e-2\qquad 2.18}                                                                  & \multicolumn{2}{l|}{4.44e-2\qquad 2.18}                                                                                                         \\ \cline{2-11} 
& 0.2221                                  &      22041                                       & \multicolumn{2}{l|}{1.03e-2\qquad2.11}                                                                        & \multicolumn{2}{l|}{1.03e-2\qquad2.11}                                                                      & \multicolumn{2}{l|}{1.03e-2\qquad2.11}                                                                  & \multicolumn{2}{l|}{1.03e-2\qquad2.11}                                                                                                         \\ \cline{2-11} 
& 0.1777                                  &      35001                                       & \multicolumn{2}{l|}{6.50e-3\qquad 2.06}                                                                        & \multicolumn{2}{l|}{6.50e-3\qquad 2.06}                                                                      & \multicolumn{2}{l|}{6.51e-3\qquad 2.06}                                                                  & \multicolumn{2}{l|}{6.50e-3\qquad 2.06}                                                                                                         \\ \hline
\multirow{4}{*}{1}                       & 0.8886                                  &    3001                                         & \multicolumn{2}{l|}{2.04e-2\qquad ---}                                                                        & \multicolumn{2}{l|}{2.04e-2\qquad ---}                                                                      & \multicolumn{2}{l|}{2.04e-2\qquad ---}                                                                  & \multicolumn{2}{l|}{2.04e-2\qquad ---}                                                                                                         \\ \cline{2-11} 
& 0.4443                                  &         12001                                    & \multicolumn{2}{l|}{3.05e-3\qquad 2.74}                                                                        & \multicolumn{2}{l|}{3.05e-3\qquad 2.74}                                                                      & \multicolumn{2}{l|}{3.05e-3\qquad 2.74}                                                                  & \multicolumn{2}{l|}{3.05e-3\qquad 2.74}                                                                                                         \\ \cline{2-11} 
& 0.2221                                  &        48001                                     & \multicolumn{2}{l|}{3.98e-4\qquad 2.94}                                                                        & \multicolumn{2}{l|}{3.98e-4\qquad 2.94}                                                                      & \multicolumn{2}{l|}{3.99e-4\qquad 2.93}                                                                  & \multicolumn{2}{l|}{3.98e-4\qquad 2.94}                                                                                                         \\ \cline{2-11} 
& 0.1777                                  &     75001                                        & \multicolumn{2}{l|}{2.04e-4\qquad 2.99}                                                                        & \multicolumn{2}{l|}{2.04e-4\qquad 2.99}                                                                      & \multicolumn{2}{l|}{2.05e-4\qquad 2.99}                                                                  & \multicolumn{2}{l|}{2.04e-4\qquad 2.99}                                                                                                         \\ \hline
\multirow{4}{*}{2}                       & 0.8886                                  &      5201                                       & \multicolumn{2}{l|}{1.40e-3\qquad ---}                                                                        & \multicolumn{2}{l|}{1.40e-3\qquad ---}                                                                      & \multicolumn{2}{l|}{1.40e-3\qquad ---}                                                                  & \multicolumn{2}{l|}{1.40e-3\qquad ---}                                                                                                         \\ \cline{2-11} 
& 0.4443                                  &     20801                                        & \multicolumn{2}{l|}{8.06e-5\qquad 4.12}                                                                        & \multicolumn{2}{l|}{8.05e-5\qquad 4.12}                                                                      & \multicolumn{2}{l|}{8.10e-5\qquad 4.11}                                                                  & \multicolumn{2}{l|}{8.05e-5\qquad 4.12}                                                                                                         \\ \cline{2-11} 
& 0.2221                                  &       83201                                      & \multicolumn{2}{l|}{4.94e-6\qquad 4.03}                                                                        & \multicolumn{2}{l|}{4.93e-6\qquad 4.03}                                                                      & \multicolumn{2}{l|}{4.97e-6\qquad 4.03}                                                                  & \multicolumn{2}{l|}{4.93e-6\qquad 4.03}                                                                                                         \\ \cline{2-11} 
& 0.1777                                  &    130001                                         & \multicolumn{2}{l|}{2.01e-6\qquad 4.02}                                                                        & \multicolumn{2}{l|}{2.01e-6\qquad 4.02}                                                                      & \multicolumn{2}{l|}{2.03e-6\qquad 4.01}                                                                  & \multicolumn{2}{l|}{2.01e-6\qquad 4.02}                                                                                                         \\ \hline
\end{tabular}
\caption{Taylor-Green vortex: Comparison of the velocities calculated with Scheme \textsf{dG2} and the scheme in \cite{akbas2018analogue,Ern2012} when $\nu=0.01$ and $\gamma=1000.0$ at $t = 0.5s$}
\label{ch_2_Taylor}
\end{table}

From Table \ref{ch_2_Taylor}, we observe that all schemes achieve the expected order of convergence and roughly the same level of accuracy.

Next, we study the Kovasznay flow with the analytical solution given by \cite{Ern2012}
\begin{align*}
\label{kovasznay}
    &\ubold(t,\xbold)=\bigg(1-\text{e}^{\lambda x_1}\cos(2\pi x_2),\frac{\lambda}{2\pi}\text{e}^{\lambda x_1}\sin(2\pi x_2)\bigg), \\
    &p(t,\xbold)=-\frac{1}{2}\text{e}^{2\lambda  x_1}-\frac{1}{8\lambda}\bigg(\text{e}^{- \lambda}-\text{e}^{3\lambda}\bigg)
\end{align*}
Here $\lambda=\frac{1}{2\nu}-(\frac{1}{4\nu^2}+4\pi^{2})^{\frac{1}{2}}$ and $\nu=0.025$. The simulation domain is set to be  $\Omega:=[-0.5,0]\times[1.5,2]$. All schemes are tested with $\gamma=1000.0$ and polynomial degree $k \in\{0,1,2\}$ on uniform meshes with mesh size $h_{\text{max}}\in\{ 0.2828, 0.1414, 0.0707, 0.0566\}$. 
\begin{table}[]
\begin{tabular}{|l|l|l|l|l|l|l|l|l|l|l|}
\hline
\multicolumn{1}{|c|}{\multirow{2}{*}{$k$}} & \multicolumn{1}{c|}{\multirow{2}{*}{$h_{\text{max}}$}} & \multicolumn{1}{c|}{\multirow{2}{*}{d.o.f}} & \multicolumn{6}{c|}{Our schemes $\|\ubold-\ubold_{h}\|_{L^{2}(\Omega)}$}                                                                                                                                                                                                                                                    & \multicolumn{2}{c|}{\multirow{2}{*}{\begin{tabular}[c]{@{}c@{}}Scheme in \cite{akbas2018analogue,Ern2012}\\ $\|\ubold-\ubold_{h}\|_{L^{2}(\Omega)}$\\ Error\qquad   order\end{tabular}}} \\ \cline{4-9}
\multicolumn{1}{|c|}{}                   & \multicolumn{1}{c|}{}                   & \multicolumn{1}{c|}{}                       & \multicolumn{2}{c|}{\begin{tabular}[c]{@{}c@{}}$\theta=0$\\ Error\qquad order\end{tabular}} & \multicolumn{2}{c|}{\begin{tabular}[c]{@{}c@{}}$\theta=1$\\ Error\qquad   order\end{tabular}} & \multicolumn{2}{c|}{\begin{tabular}[c]{@{}c@{}}$\theta=-1$\\ Error\qquad   order\end{tabular}} & \multicolumn{2}{c|}{}                                                                                                         \\ \hline
\multirow{4}{*}{0}                       & 0.2828                                  &    1401                                         & \multicolumn{2}{l|}{9.26e-2\qquad ---}                                                                        & \multicolumn{2}{l|}{9.26e-2\qquad ---}                                                                      & \multicolumn{2}{l|}{9.27e-2\qquad ---}                                                                  & \multicolumn{2}{l|}{1.49e-1\qquad ---}                                                                                                         \\ \cline{2-11} 
& 0.1414                                  &      5601                                       & \multicolumn{2}{l|}{2.40e-2\qquad 1.95}                                                                        & \multicolumn{2}{l|}{2.40e-2\qquad 1.95}                                                                      & \multicolumn{2}{l|}{2.40e-2\qquad 1.95}                                                                  & \multicolumn{2}{l|}{2.98e-2\qquad 2.33}                                                                                                         \\ \cline{2-11} 
& 0.0707                                  &      22041                                       & \multicolumn{2}{l|}{6.13e-3\qquad1.97}                                                                        & \multicolumn{2}{l|}{6.13e-3\qquad1.97}                                                                      & \multicolumn{2}{l|}{6.14e-3\qquad1.97}                                                                  & \multicolumn{2}{l|}{6.53e-3\qquad2.19}                                                                                                         \\ \cline{2-11} 
& 0.0566                                  &      35001                                       & \multicolumn{2}{l|}{3.95e-3\qquad 1.97}                                                                        & \multicolumn{2}{l|}{3.95e-3\qquad 1.97}                                                                      & \multicolumn{2}{l|}{3.95e-3\qquad 1.97}                                                                  & \multicolumn{2}{l|}{4.11e-3\qquad 2.08}                                                                                                         \\ \hline
\multirow{4}{*}{1}                       & 0.2828                                  &    3001                                         & \multicolumn{2}{l|}{1.08e-2\qquad ---}                                                                        & \multicolumn{2}{l|}{1.08e-2\qquad ---}                                                                      & \multicolumn{2}{l|}{1.08e-2\qquad ---}                                                                  & \multicolumn{2}{l|}{1.12e-2\qquad ---}                                                                                                         \\ \cline{2-11} 
& 0.1414                                  &         12001                                    & \multicolumn{2}{l|}{1.35e-3\qquad 3.01}                                                                        & \multicolumn{2}{l|}{1.35e-3\qquad 3.01}                                                                      & \multicolumn{2}{l|}{1.35e-3\qquad 3.01}                                                                  & \multicolumn{2}{l|}{1.35e-3\qquad 3.06}                                                                                                         \\ \cline{2-11} 
& 0.0707                                  &        48001                                     & \multicolumn{2}{l|}{1.68e-4\qquad 3.00}                                                                        & \multicolumn{2}{l|}{1.68e-4\qquad 3.00}                                                                      & \multicolumn{2}{l|}{1.68e-4\qquad 3.00}                                                                  & \multicolumn{2}{l|}{1.68e-4\qquad 3.00}                                                                                                         \\ \cline{2-11} 
& 0.0566                                  &     75001                                        & \multicolumn{2}{l|}{8.58e-5\qquad 3.00}                                                                        & \multicolumn{2}{l|}{8.57e-5\qquad 3.00}                                                                      & \multicolumn{2}{l|}{8.58e-5\qquad 3.00}                                                                  & \multicolumn{2}{l|}{8.57e-5\qquad 3.00}                                                                                                         \\ \hline
\multirow{4}{*}{2}                       & 0.2828                                  &      5201                                       & \multicolumn{2}{l|}{8.93e-4\qquad ---}                                                                        & \multicolumn{2}{l|}{8.93e-4\qquad ---}                                                                      & \multicolumn{2}{l|}{8.93e-4\qquad ---}                                                                  & \multicolumn{2}{l|}{8.92e-4\qquad ---}                                                                                                         \\ \cline{2-11} 
& 0.1414                                  &     20801                                        & \multicolumn{2}{l|}{5.82e-5\qquad 3.94}                                                                        & \multicolumn{2}{l|}{5.82e-5\qquad 3.94}                                                                      & \multicolumn{2}{l|}{5.82e-5\qquad 3.94}                                                                  & \multicolumn{2}{l|}{5.82e-5\qquad 3.94}                                                                                                         \\ \cline{2-11} 
& 0.0707                                  &       83201                                      & \multicolumn{2}{l|}{3.73e-6\qquad 3.96}                                                                        & \multicolumn{2}{l|}{3.73e-6\qquad 3.96}                                                                      & \multicolumn{2}{l|}{3.73e-6\qquad 3.96}                                                                  & \multicolumn{2}{l|}{3.73e-6\qquad 3.96}                                                                                                         \\ \cline{2-11} 
& 0.0566                                  &    130001                                         & \multicolumn{2}{l|}{1.55e-6\qquad 3.95}                                                                        & \multicolumn{2}{l|}{1.55e-6\qquad 3.95}                                                                      & \multicolumn{2}{l|}{1.55e-6\qquad 3.95}                                                                  & \multicolumn{2}{l|}{1.55e-6\qquad 3.95}                                                                                                         \\ \hline
\end{tabular}
\caption{Kovasznay Flow: Comparison of the velocities calculated with Scheme \textsf{dG2} and the scheme  in \cite{akbas2018analogue,Ern2012} when $\nu=0.025$ and $\gamma=1000.0$}
\label{ch_2_Kflow}
\end{table}
From Table \ref{ch_2_Kflow}, we observe that all four schemes have the expected order of convergence. The three versions of our scheme achieve roughly the same level of accuracy, and relatively smaller absolute errors compared with scheme in \cite{Ern2012} especially when $k=0$. 

In summary, the consistent formulation \textsf{dG2} is comparable to classical dG schemes in the literature.
\subsection{Temporal Accuracy Test}
In this subsection, we use the following two-dimensional unsteady flow to investigate the temporal accuracy of our schemes \cite{kim2000second}
\begin{align*}
    &\ubold(t,\xbold)=\bigg(-\cos(\pi x_{1})\sin(\pi x_{2})e^{-2\pi ^2 \nu t},\sin(\pi x_{1})\cos(\pi x_{2})e^{-2\pi^2\nu t}\bigg), \\ &p(t,\xbold)=-\frac{1}{4}\bigg(\cos(2\pi x_{1})+\cos(2\pi x_{2})\bigg)e^{-4\pi ^2\nu t}.
\end{align*}
The computational domain is $\Omega:=[-0.5,-0.5]\times[0.5,0.5]$ and $\nu$ is set to be $0.1$. We focus on the \textsf{dG2} with $\theta=1.0$ and $\gamma=1000.0$ for simplicity of presentation. Simulations are performed with polynomial degree $k\in\{1,2\}$, mesh size $h_{\text{max}}=0.0283$, time steps $\tau\in\{ 0.01s, 0.02s, 0.04s,0.08s\}$, and the total simulation time $T=0.8s$. It could be observed from Table \ref{ch_2_temporal} that the Crank-Nicolson time discretization has reasonable  second order accuracy.
\begin{table}[]
\begin{tabular}{|c|l|l|l|l|l|}
\hline
$k$                  & \multicolumn{1}{c|}{$\tau$}    & \multicolumn{2}{c|}{$\|\ubold-\ubold_h\|_{L^{2}(\Omega)}$} & \multicolumn{2}{c|}{Order} \\ \hline
\multirow{4}{*}{1} & \multicolumn{1}{c|}{0.01s} & \multicolumn{2}{c|}{1.62e-5}      & \multicolumn{2}{c|}{---}      \\ \cline{2-6} 
& \multicolumn{1}{c|}{0.02s} & \multicolumn{2}{c|}{6.42e-5}      & \multicolumn{2}{c|}{1.99}      \\ \cline{2-6} 
& 0.04s                      & \multicolumn{2}{c|}{2.55e-4}      & \multicolumn{2}{c|}{1.99}      \\ \cline{2-6} 
& 0.08s                      & \multicolumn{2}{c|}{1.19e-3}      & \multicolumn{2}{c|}{2.23}      \\ \hline
\multirow{4}{*}{2} & 0.01s                      & \multicolumn{2}{c|}{1.61e-5}      & \multicolumn{2}{c|}{---}      \\ \cline{2-6} 
& 0.02s                      & \multicolumn{2}{c|}{6.41e-5}      & \multicolumn{2}{c|}{2.00}      \\ \cline{2-6} 
& 0.04s                      & \multicolumn{2}{c|}{2.55e-4}      & \multicolumn{2}{c|}{1.99}      \\ \cline{2-6} 
& 0.08s                      & \multicolumn{2}{c|}{1.19e-3}      & \multicolumn{2}{c|}{2.23}      \\ \hline
\end{tabular}
\caption{Temporal accuracy test with Scheme \textsf{dG2} when $\nu=0.1$, $\theta=1.0$, $h_{\text{max}}=0.0283$ and $\gamma=1000.0$}
\label{ch_2_temporal}
\end{table}
\subsection{Accuracy of Scheme \textsf{dG1}}\label{dG1_1}
In this subsection, we use Taylor-Green vortex and Kovasznay problems to study the behavior of the \textsf{dG1} formulation with respect to $\gamma\in\{1,10^3,10^6\}$. For simplicity, we focus on the case when $\theta=1.0$. The basic setting is as in Subsection \ref{c_2h_test}. Note that $\gamma$ is used to increase pressure robustness in the consistent formulation \textsf{dG2}. However,  \textsf{dG1} is not consistent and $\gamma$ is additionally used for weakly enforcing consistency. The main purpose here is to study the numerical influence of $\gamma$ on \textsf{dG1}.

In Table \ref{c_1h_Taylor}, we observe a decrease of absolute velocity errors when $\gamma$ increases, and optimal convergence when $\gamma=10^{6}$ for the Taylor-Green vortex. It can be observed from Figure~\ref{ch1_TGV} that weird contours of vorticity when $\gamma=1.0$ are cured by increasing $\gamma$ to $\gamma=10^3$ and $10^6$. The improved numerical performance of \textsf{dG1} as $\gamma$ grows is due to the enhanced consistency and pressure robustness by larger $\gamma$.
Further comparison between the corresponding columns in Tables \ref{ch_2_Taylor} and \ref{c_1h_Taylor} indicates that a larger $\gamma$ is needed to maintain the same level of accuracy as in the \textsf{dG2} formulation, especially when $k=2$.

In the case of the Kovasznay flow, the \textsf{dG1} and \textsf{dG2} formulations achieve roughly the same level of accuracy with $\gamma=1000.0$ for all values of $k$. In the high resolution case $h_{\max}\in\{0.0707,0.0566\}$, $k=2$,  $\gamma=10^{6}$, there is an abnormal reduction on the rate of convergence. We point out that this phenomenon could be fixed by further reducing the absolute and relative tolerances to $10^{-12}$ in the Newton solver. A possible explanation is that the condition number of the stiffness matrix with $\gamma=10^{6}$ of the linearized equations is significantly larger than stiffness matrices with $\gamma\in\{1,10^3\}$, and thus smaller iterative error tolerances are needed to recover numerical accuracy. When using coarse meshes and lower order polynomials (common in practice), we still recommended to use large $\gamma$. The observed different numerical behaviors with respect to $\gamma$ in steady and unsteady cases shed some light in the difference between stationary and evolutionary problems. 
\begin{table}[]
\begin{tabular}{|l|l|l|l|l|l|l|l|l|}
\hline
\multicolumn{1}{|c|}{\multirow{2}{*}{$k$}} & \multicolumn{1}{c|}{\multirow{2}{*}{$h_{\text{max}}$}} & \multicolumn{1}{c|}{\multirow{2}{*}{d.o.f}} & \multicolumn{6}{c|}{$\|\ubold-\ubold_h\|_{L^{2}(\Omega)}$}                                                                                                                                                                                                                                                    \\ \cline{4-9} 
\multicolumn{1}{|c|}{}                   & \multicolumn{1}{c|}{}                   & \multicolumn{1}{c|}{}                       & \multicolumn{2}{c|}{\begin{tabular}[c]{@{}c@{}}$\gamma=1.0$\\ Error\qquad   order\end{tabular}} & \multicolumn{2}{c|}{\begin{tabular}[c]{@{}c@{}}$\gamma=10^{3}$\\ Error\qquad  order\end{tabular}} & \multicolumn{2}{c|}{\begin{tabular}[c]{@{}c@{}}$\gamma=10^{6}$\\ Error\qquad   order\end{tabular}} \\ \hline
\multirow{4}{*}{0}                       & 0.8886                                  &         1401                                    & \multicolumn{2}{l|}{4.89e-1\qquad ---}                                                                        & \multicolumn{2}{l|}{2.02e-1\qquad ---}                                                                      & \multicolumn{2}{l|}{1.82e-1\qquad ---}                                                                  \\ \cline{2-9} 
& 0.4443                                  &          5601                                   & \multicolumn{2}{l|}{2.73e-1\qquad 8.42e-1}                                                                        & \multicolumn{2}{l|}{4.44e-2\qquad 2.18}                                                                      & \multicolumn{2}{l|}{4.35e-2\qquad 2.06}                                                                  \\ \cline{2-9} 
& 0.2221                                  &     22401                                        & \multicolumn{2}{l|}{1.85e-1\qquad 5.65e-1}                                                                        & \multicolumn{2}{l|}{1.03e-2\qquad 2.11}                                                                      & \multicolumn{2}{l|}{1.02e-2\qquad 2.09}                                                                  \\ \cline{2-9} 
& 0.1777                                  &     35001                                        & \multicolumn{2}{l|}{1.65e-1\qquad 5.00e-1}                                                                        & \multicolumn{2}{l|}{6.51e-3\qquad 2.06}                                                                      & \multicolumn{2}{l|}{6.48e-3\qquad 2.05}                                                                  \\ \hline
\multirow{4}{*}{1}                       & 0.8886                                  &      3001                                       & \multicolumn{2}{l|}{2.36e-1\qquad ---}                                                                        & \multicolumn{2}{l|}{2.04e-2\qquad ---}                                                                      & \multicolumn{2}{l|}{2.00e-2\qquad ---}                                                                  \\ \cline{2-9} 
& 0.4443                                  &      12001                                       & \multicolumn{2}{l|}{1.74e-1\qquad 4.42e-1}                                                                        & \multicolumn{2}{l|}{3.03e-3\qquad 2.75}                                                                      & \multicolumn{2}{l|}{3.04e-3\qquad 2.72}                                                                  \\ \cline{2-9} 
& 0.2221                                  &      48001                                       & \multicolumn{2}{l|}{1.25e-1\qquad 4.70e-1}                                                                        & \multicolumn{2}{l|}{4.37e-4\qquad 2.79}                                                                      & \multicolumn{2}{l|}{3.98e-4\qquad 2.93}                                                                  \\ \cline{2-9} 
& 0.1777                                  &   75001                                          & \multicolumn{2}{l|}{1.13e-1\qquad 4.82e-1}                                                                        & \multicolumn{2}{l|}{2.70e-4\qquad 2.16}                                                                      & \multicolumn{2}{l|}{2.04e-4\qquad 2.99}                                                                  \\ \hline
\multirow{4}{*}{2}                       & 0.8886                                  &    5201                                         & \multicolumn{2}{l|}{1.75e-1\qquad ---}                                                                        & \multicolumn{2}{l|}{1.42e-3\qquad ---}                                                                      & \multicolumn{2}{l|}{1.39e-3\qquad ---}                                                                  \\ \cline{2-9} 
& 0.4443                                  &    20801                                         & \multicolumn{2}{l|}{1.26e-1\qquad 4.70e-1}                                                                        & \multicolumn{2}{l|}{2.21e-4\qquad 2.68}                                                                      & \multicolumn{2}{l|}{8.05e-5\qquad 4.11}                                                                  \\ \cline{2-9} 
& 0.2221                                  &     83201                                        & \multicolumn{2}{l|}{9.02e-2\qquad 4.85e-1}                                                                        & \multicolumn{2}{l|}{1.48e-4\qquad 5.79e-1}                                                                      & \multicolumn{2}{l|}{4.93e-6\qquad 4.03}                                                                  \\ \cline{2-9} 
& 0.1777                                  &    130001                                         & \multicolumn{2}{l|}{8.09e-2\qquad 4.90e-1}                                                                        & \multicolumn{2}{l|}{1.33e-4\qquad 4.91e-1}                                                                      & \multicolumn{2}{l|}{2.01e-6\qquad 4.02}                                                                  \\ \hline
\end{tabular}
\caption{Taylor-Green vortex: Performance study of Scheme \textsf{dG1} when $\nu=0.01$ at $t = 0.5s$}
\label{c_1h_Taylor}
\end{table}
\begin{figure*}
\centering
\begin{multicols}{2}
    \includegraphics[width=0.75\linewidth,trim={7cm 2cm 0cm 2cm},clip]{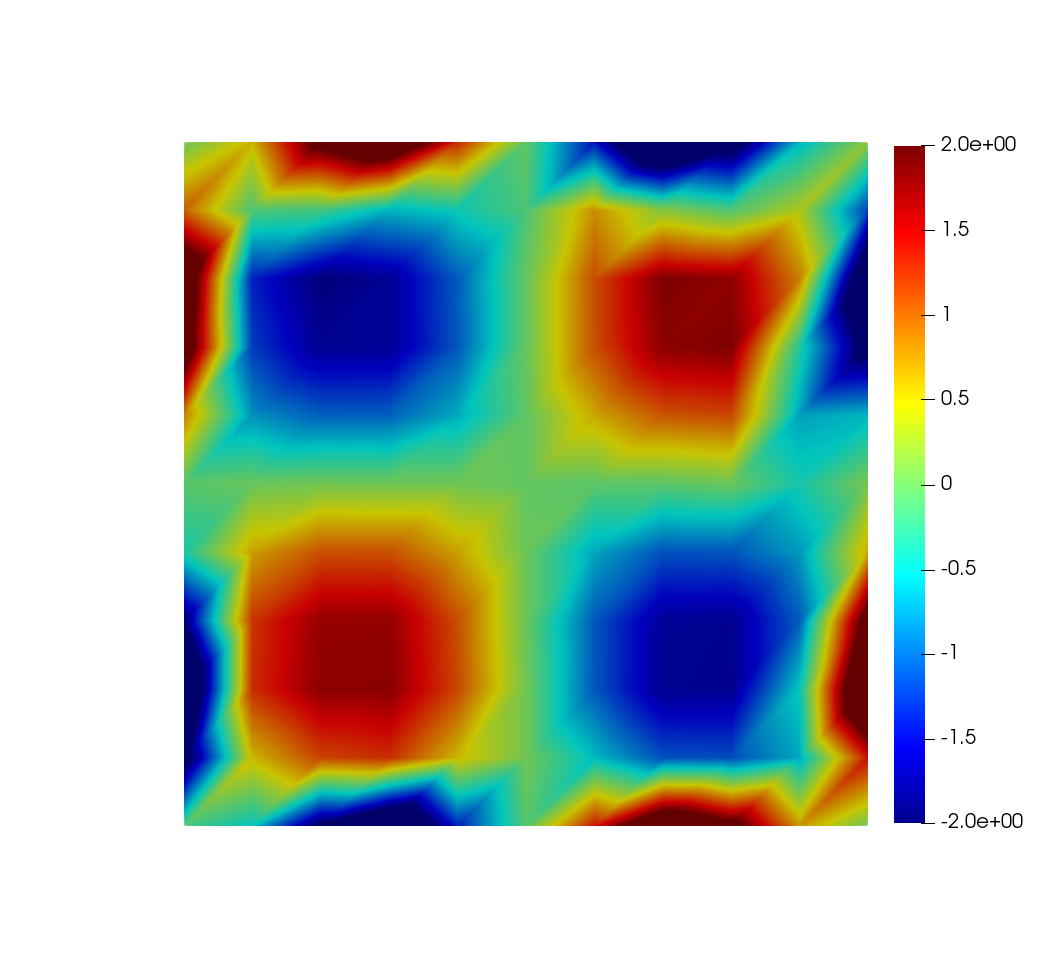}\par
    \includegraphics[width=0.75\linewidth,trim={7cm 2cm 0cm 2cm},clip]{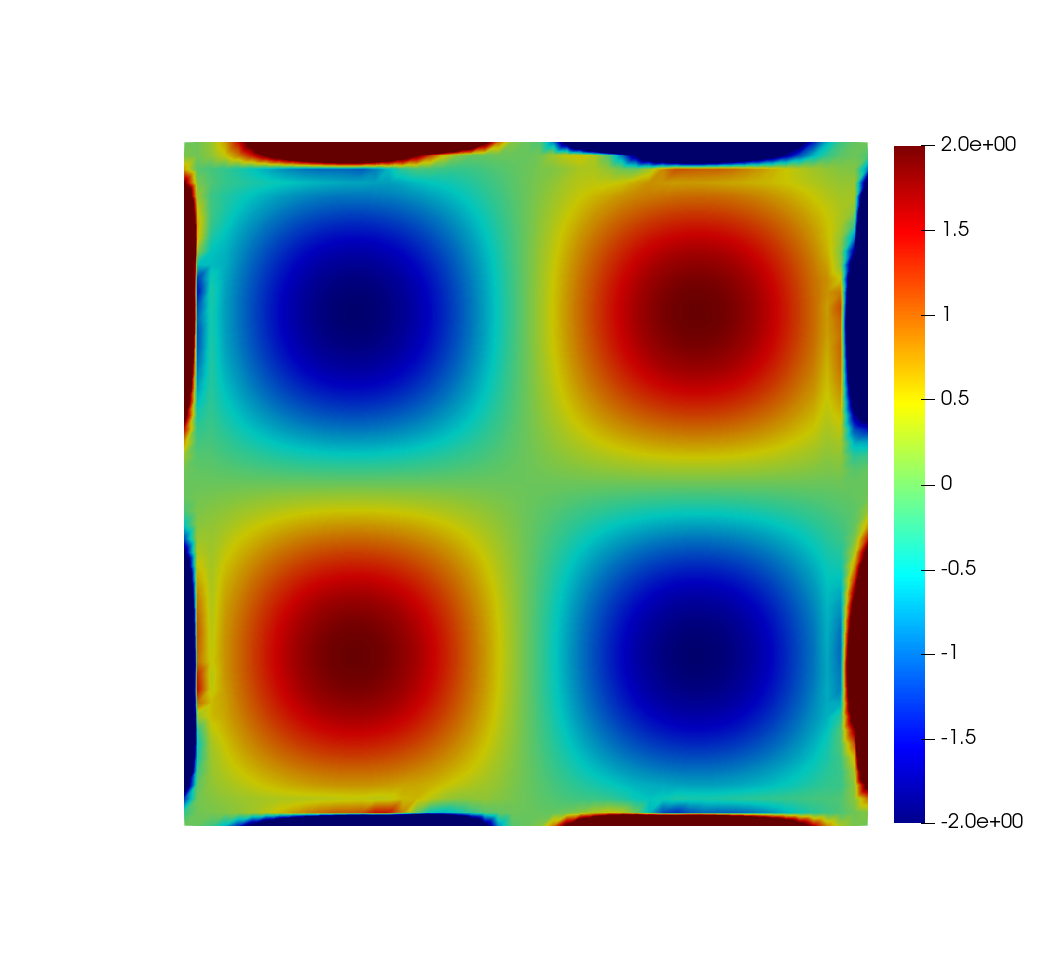}\par 
    \end{multicols}
\begin{multicols}{2}
    \includegraphics[width=0.75\linewidth,trim={7cm 2cm 0cm 2cm},clip]{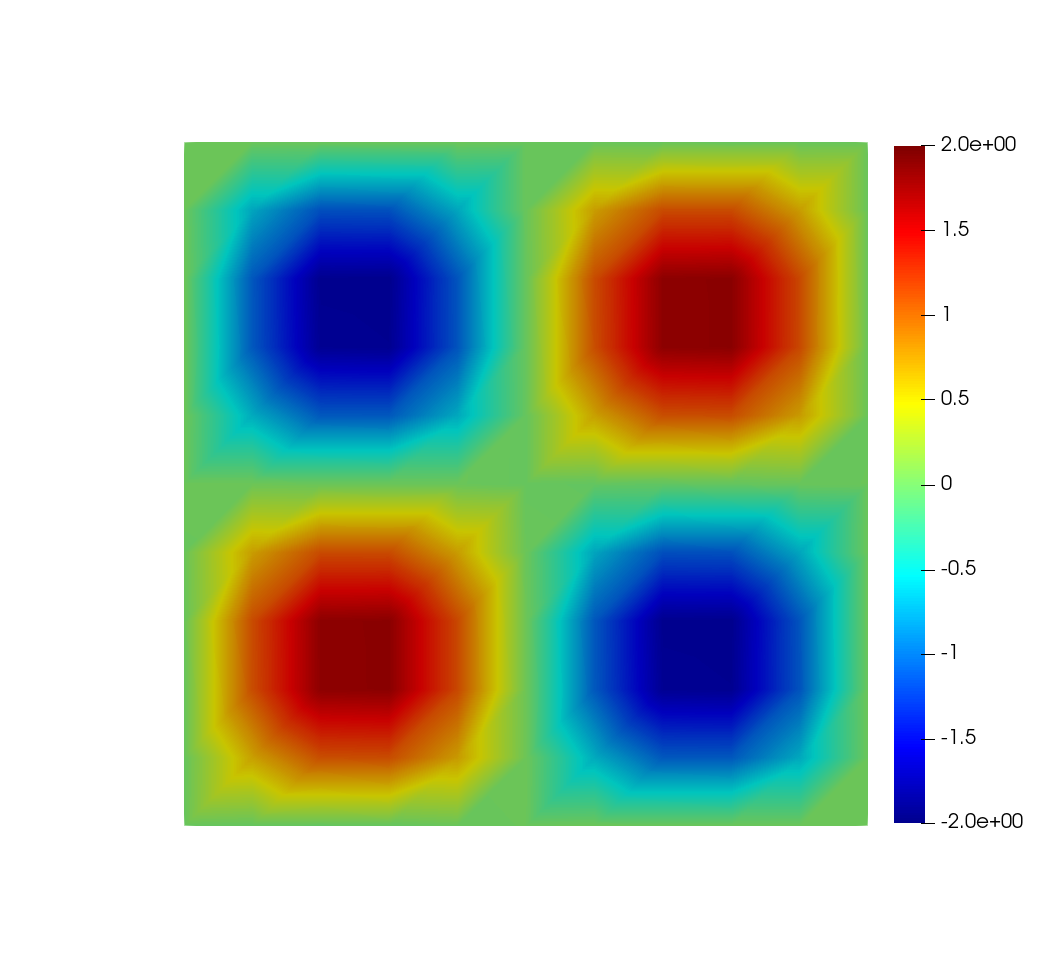}\par
    \includegraphics[width=0.75\linewidth,trim={7cm 2cm 0cm 2cm},clip]{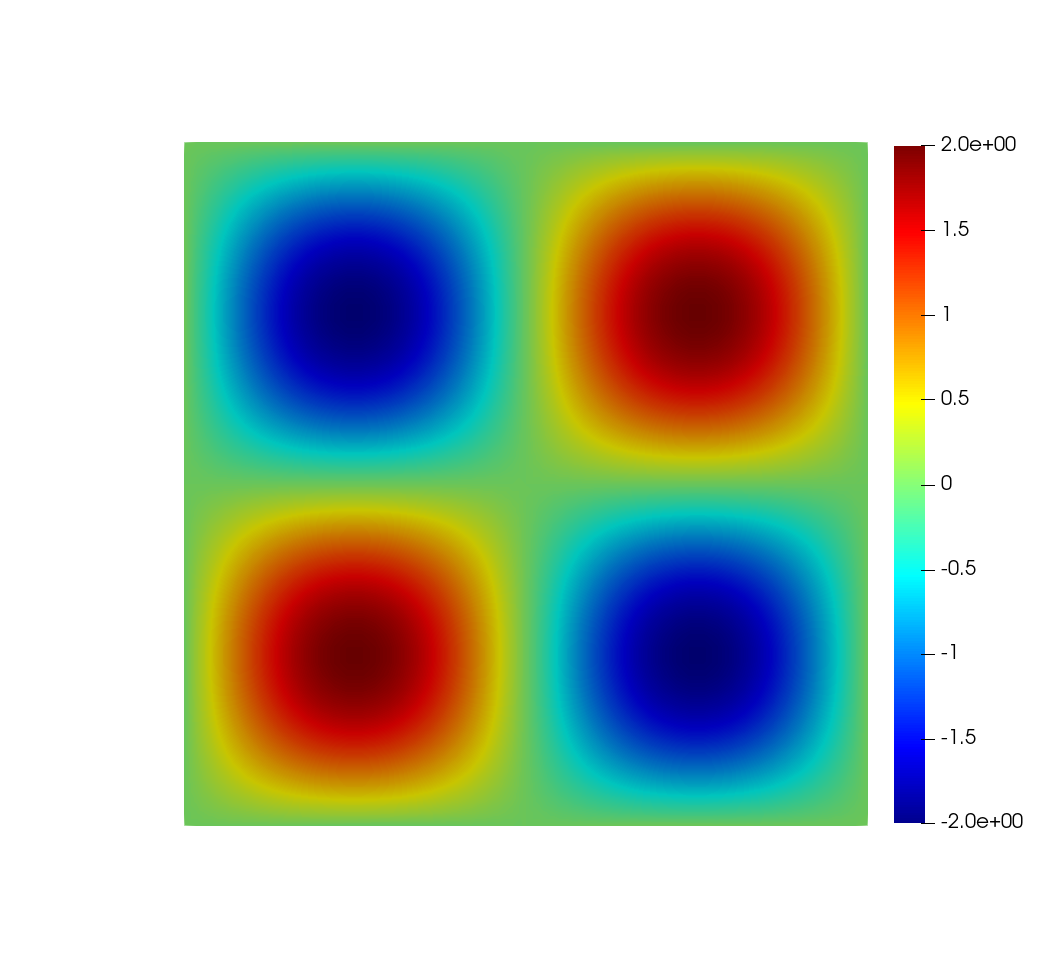}\par
\end{multicols}
\begin{multicols}{2}
    \includegraphics[width=0.75\linewidth,trim={7cm 2cm 0cm 2cm},clip]{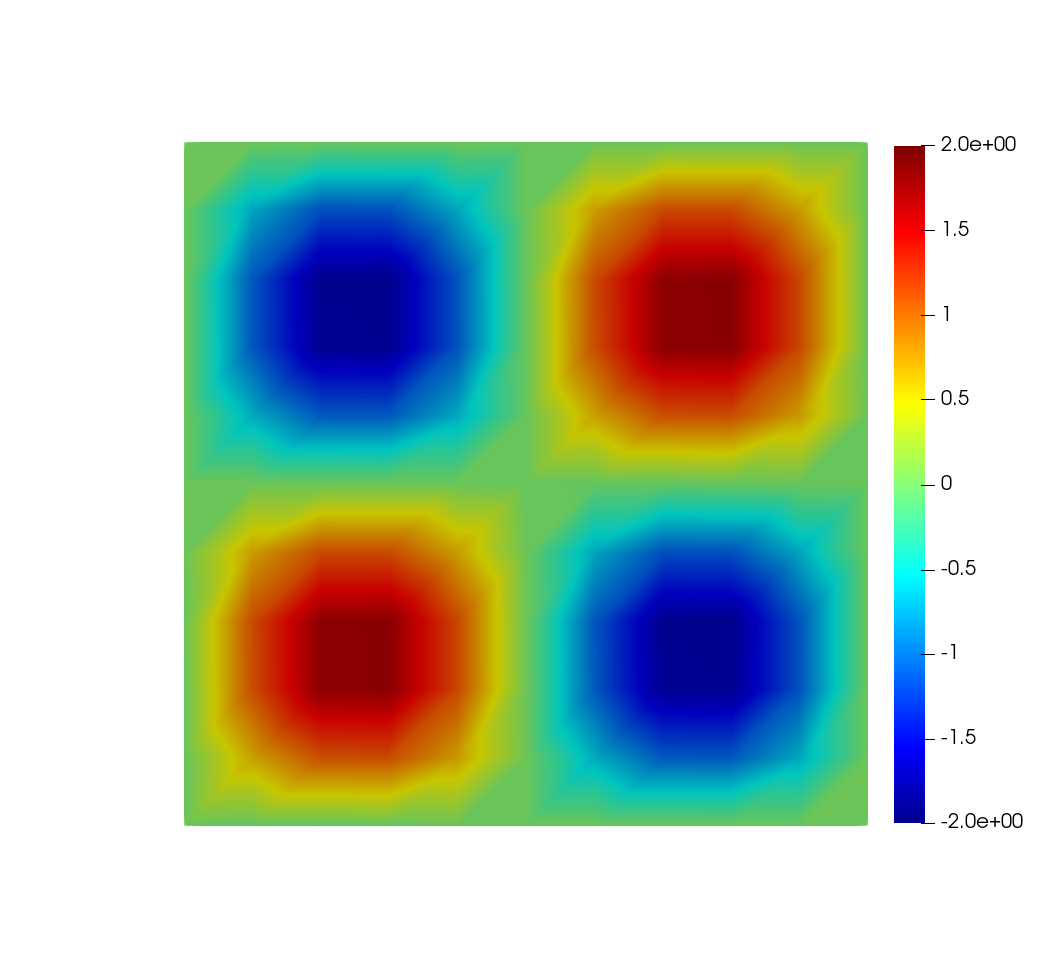}\par
    \includegraphics[width=0.75\linewidth,trim={7cm 2cm 0cm 2cm},clip]{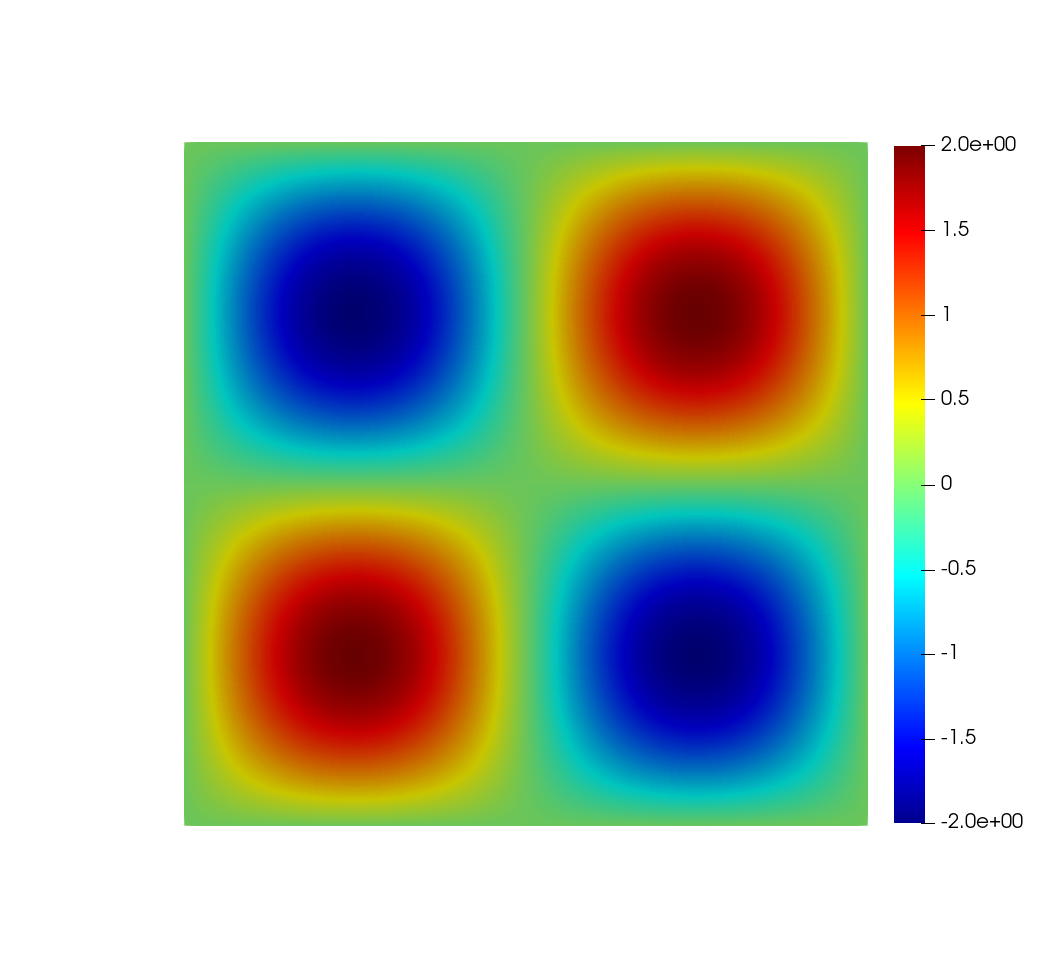}\par
\end{multicols}
\caption{Taylor-Green vortex, \textsf{dG1}: Contours of vorticity with $\gamma=1$ (upper row), $\gamma=10^{3}$ (middle row) and $\gamma=10^{6}$ (lower row) when $h_{\text{max}}=0.8886$ (left) and $h_{\text{max}}=0.1777$ (right) at $t=0.5s$ with $\nu=0.01$ and $k=2$}
\label{ch1_TGV}
\end{figure*}
\begin{table}
\begin{center}
\begin{tabular}{|l|l|l|l|l|l|l|l|l|}
\hline
\multicolumn{1}{|c|}{\multirow{2}{*}{$k$}} & \multicolumn{1}{c|}{\multirow{2}{*}{$h_{\text{max}}$}} & \multicolumn{1}{c|}{\multirow{2}{*}{d.o.f}} & \multicolumn{6}{c|}{$\|\ubold-\ubold_h\|_{L^{2}(\Omega)}$}                                                                                                                                                                                                                                                    \\ \cline{4-9} 
\multicolumn{1}{|c|}{}                   & \multicolumn{1}{c|}{}                   & \multicolumn{1}{c|}{}                       & \multicolumn{2}{c|}{\begin{tabular}[c]{@{}c@{}}$\gamma=1.0$\\ Error\qquad   order\end{tabular}} & \multicolumn{2}{c|}{\begin{tabular}[c]{@{}c@{}}$\gamma=10^{3}$\\ Error\qquad  order\end{tabular}} & \multicolumn{2}{c|}{\begin{tabular}[c]{@{}c@{}}$\gamma=10^{6}$\\ Error\qquad   order\end{tabular}} \\ \hline
\multirow{4}{*}{0}                       & 0.2828                                 &         1401                                   & \multicolumn{2}{l|}{1.40e-1\qquad ---}                                                                        & \multicolumn{2}{l|}{9.26e-2\qquad ---}                                                                      & \multicolumn{2}{l|}{9.26e-2\qquad ---}                                                                  \\ \cline{2-9} 
& 0.1414                                 &          5601                                   & \multicolumn{2}{l|}{4.35e-2\qquad 1.69}                                                                        & \multicolumn{2}{l|}{2.40e-2\qquad 1.95}                                                                      & \multicolumn{2}{l|}{2.40e-2\qquad 1.95}                                                                  \\ \cline{2-9} 
& 0.0707                                  &     22401                                        & \multicolumn{2}{l|}{1.15e-2\qquad 1.92}                                                                        & \multicolumn{2}{l|}{6.13e-3\qquad 1.97}                                                                      & \multicolumn{2}{l|}{6.23e-3\qquad 1.95}                                                                  \\ \cline{2-9} 
& 0.0566                                  &     35001                                        & \multicolumn{2}{l|}{7.44e-3\qquad 1.93}                                                                        & \multicolumn{2}{l|}{3.95e-3\qquad 1.97}                                                                      & \multicolumn{2}{l|}{4.12e-3\qquad 1.84}                                                                  \\ \hline
\multirow{4}{*}{1}                       & 0.2828                                  &      3001                                       & \multicolumn{2}{l|}{1.92e-2\qquad ---}                                                                        & \multicolumn{2}{l|}{1.08e-2\qquad ---}                                                                      & \multicolumn{2}{l|}{1.08e-2\qquad ---}                                                                  \\ \cline{2-9} 
& 0.1414                                  &      12001                                       & \multicolumn{2}{l|}{6.10e-3\qquad 1.65}                                                                        & \multicolumn{2}{l|}{1.35e-3\qquad 3.01}                                                                      & \multicolumn{2}{l|}{1.35e-3\qquad 3.01}                                                                  \\ \cline{2-9} 
& 0.0707                                  &      48001                                       & \multicolumn{2}{l|}{2.20e-3\qquad 1.47}                                                                        & \multicolumn{2}{l|}{1.68e-4\qquad 3.00}                                                                      & \multicolumn{2}{l|}{1.27e-3\qquad 8.86e-2}                                                                  \\ \cline{2-9} 
& 0.0566                                  &   75001                                          & \multicolumn{2}{l|}{1.54e-3\qquad 1.61}                                                                        & \multicolumn{2}{l|}{8.59e-5\qquad 3.00}                                                                      & \multicolumn{2}{l|}{1.26e-3\qquad 3.78e-2}                                                                  \\ \hline
\multirow{4}{*}{2}                       & 0.2828                                 &    5201                                         & \multicolumn{2}{l|}{7.71e-3\qquad ---}                                                                        & \multicolumn{2}{l|}{8.94e-4\qquad ---}                                                                      & \multicolumn{2}{l|}{8.93e-4\qquad ---}                                                                  \\ \cline{2-9} 
& 0.1414                                  &    20801                                         & \multicolumn{2}{l|}{2.52e-3\qquad 1.61}                                                                        & \multicolumn{2}{l|}{5.85e-5\qquad 3.93}                                                                      & \multicolumn{2}{l|}{5.82e-5\qquad 3.94}                                                                  \\ \cline{2-9} 
& 0.0707                                  &     83201                                        & \multicolumn{2}{l|}{8.14e-4\qquad 1.63}                                                                        & \multicolumn{2}{l|}{3.90e-6\qquad 3.91}                                                                      & \multicolumn{2}{l|}{1.25e-3\qquad -4.42}                                                                  \\ \cline{2-9} 
& 0.0566                                  &    130001                                         & \multicolumn{2}{l|}{5.72e-4\qquad 1.58}                                                                        & \multicolumn{2}{l|}{1.72e-6\qquad 3.67}                                                                      & \multicolumn{2}{l|}{1.25e-3\quad\: $<0.001$}                                                                  \\ \hline
\end{tabular}
\caption{Kovasznay Flow: Performance study of the \textsf{dG1} formulation when $\nu=0.025$}
\label{c_1h_Kflow}
\end{center}
\end{table}

\clearpage
\subsection{Flow Over a Cylinder and Lid Driven Flow by \textsf{dG1}}\label{dG1_2}
In the last subsection, we continue to study the behavior of the \textsf{dG1} formulation for the problems of flow over a cylinder \cite{layton2009accuracy,schafer1996benchmark,LangtangenLogg2017} and lid driven flow when $\theta=1.0$ and $\gamma=1000.0$. The goal is to check whether the scheme could capture the essential physical features of such flows. Recall that $\textsf{dG1}$ is not a consistent formulation by nature, and its consistency is weakly imposed through $\gamma$.

We first study the flow over a cylinder. A primary feature of this problem is the formation of von K\'arm\'an vortex street, which depends on both the spatial and temporal discretizations. Therefore the ability to have vortex properly generated is not trivial, see, e.g., \cite{xi2020,layton2009accuracy} for numerical examples failing to have this property. The domain is set to be $\Omega:=\left([0,2.2]\times[0,0.41]\right)\backslash B$ with $B$ a disk centered at $(0.2,0.2)$ and of radius $0.05$. We run the simulation for $8s$ using a time step $\tau_h=0.01s$. The flow is subjected to the following inflow and outflow profile (\cite{layton2009accuracy})
\begin{align*}
    &u(t,0,x_2)=u(t,2.2,x_2)=\frac{6}{0.41^{2}}\sin(\pi t/8)x_2(0.41-x_2),\\
    &v(t,0,x_2)=v(t,2.2,x_2)=0,
\end{align*}
and vanishing velocity on the rest of $\partial\Omega$. The viscosity is set to be $0.001$, and therefore the mean Reynolds number ranges from $0$ to $100$.

From Figure~\ref{cylinder_velocity_contour_ch1}, we observe that 
\textsf{dG1} is able to recover the vortex formation successfully.

Finally, we consider the lid driven flow on the square domain $[0,1]^2$ with a tangential velocity $\ubold=(1,0)$ on the top and vanishing velocity on the other sides. We are interested in the flow with $Re=400$, which has a large gradient near the lid due to the singularities introduced by the jump of the velocity on the top two corners, and the recovery of the small corner vortex is essential.
\begin{figure}
      \includegraphics[width=1.0\textwidth,trim={0cm 12cm 0cm 10cm},clip]{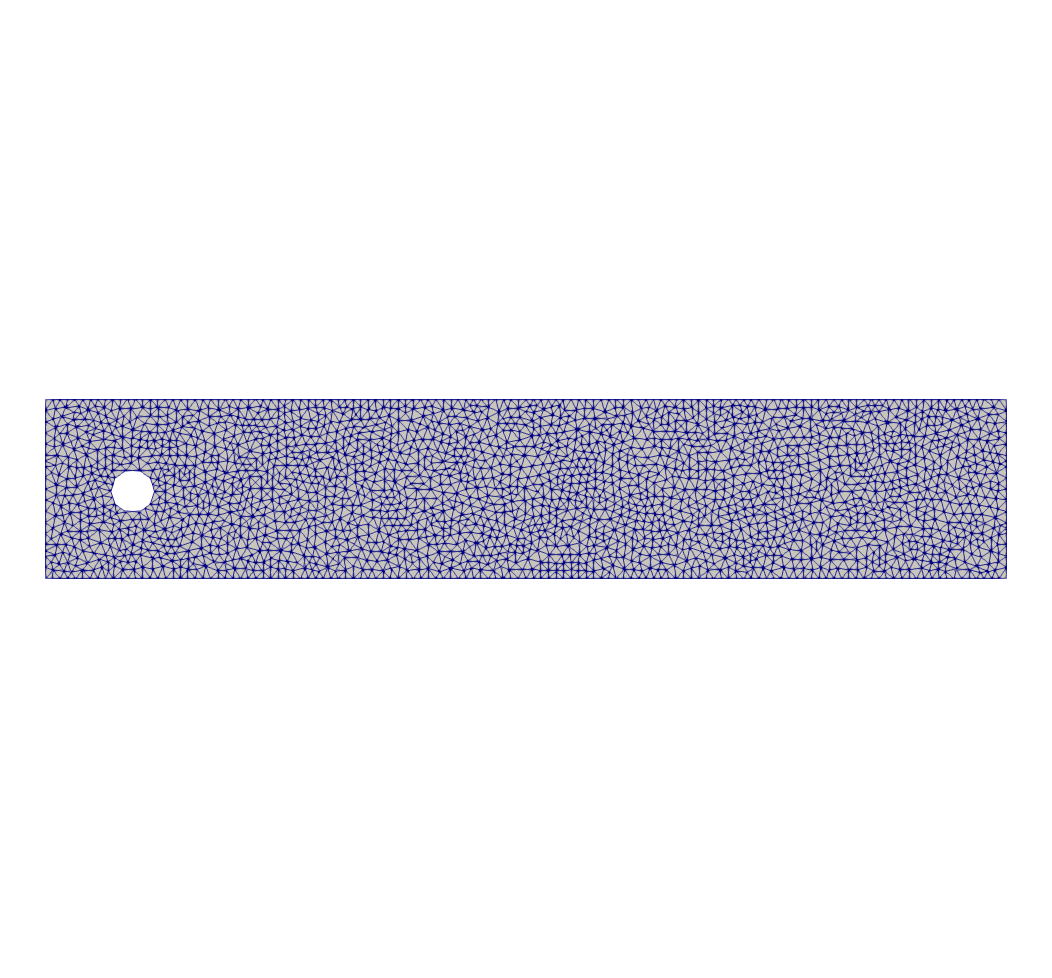}
      \vspace{4pt}
  \caption{Mesh for the flow around  a cylinder}
  \label{cylinder_mesh}
\end{figure}
\begin{figure}
      \includegraphics[width=1.0\textwidth,trim={0cm 12cm 0cm 10cm},clip]{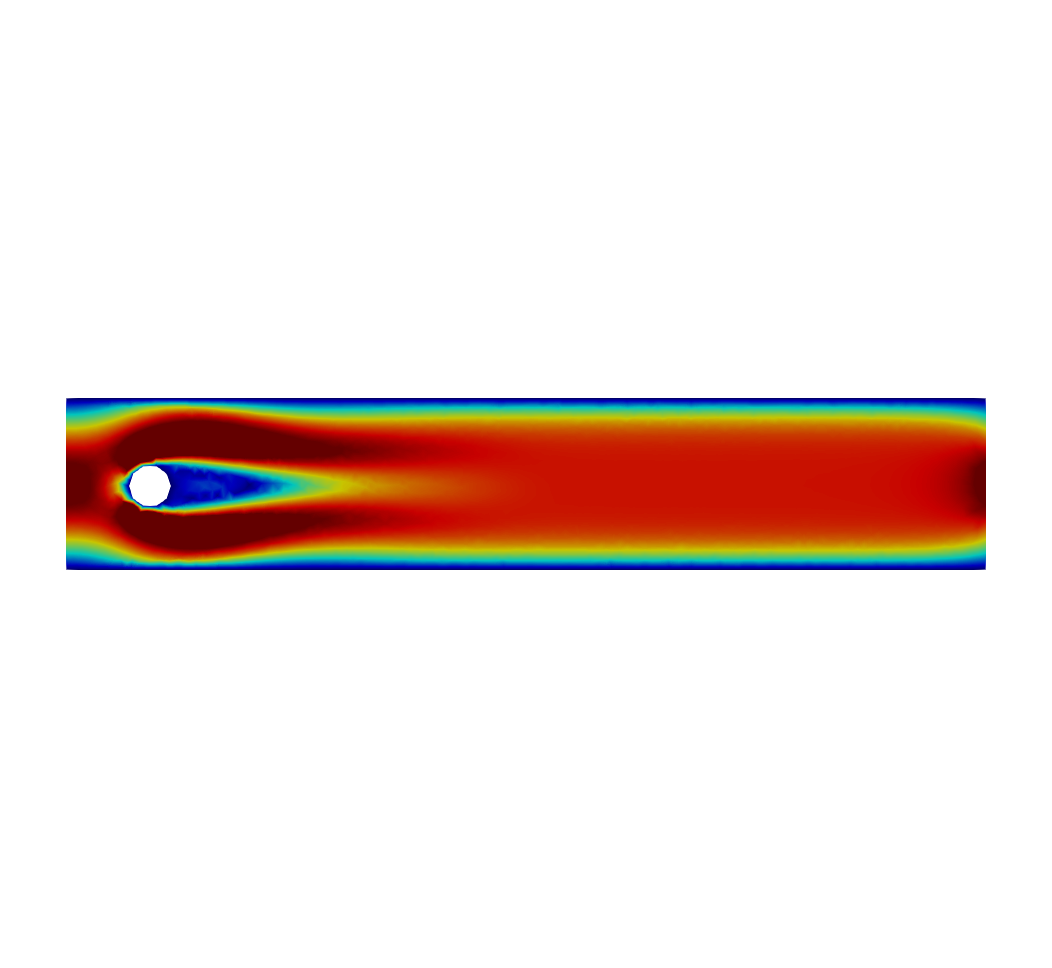}
      \includegraphics[width=1.0\textwidth,trim={0cm 12cm 0cm 10cm},clip]{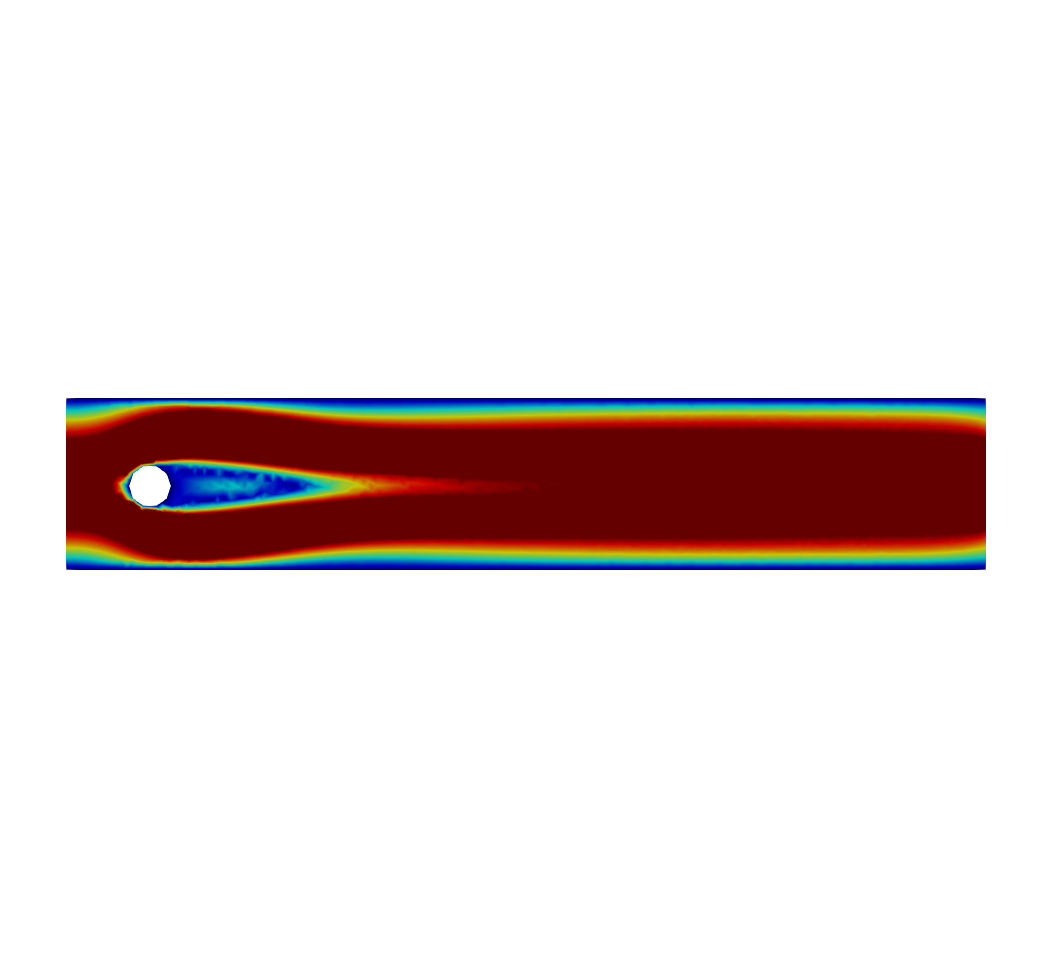}
      \includegraphics[width=1.0\textwidth,trim={0cm 12cm 0cm 10cm},clip]{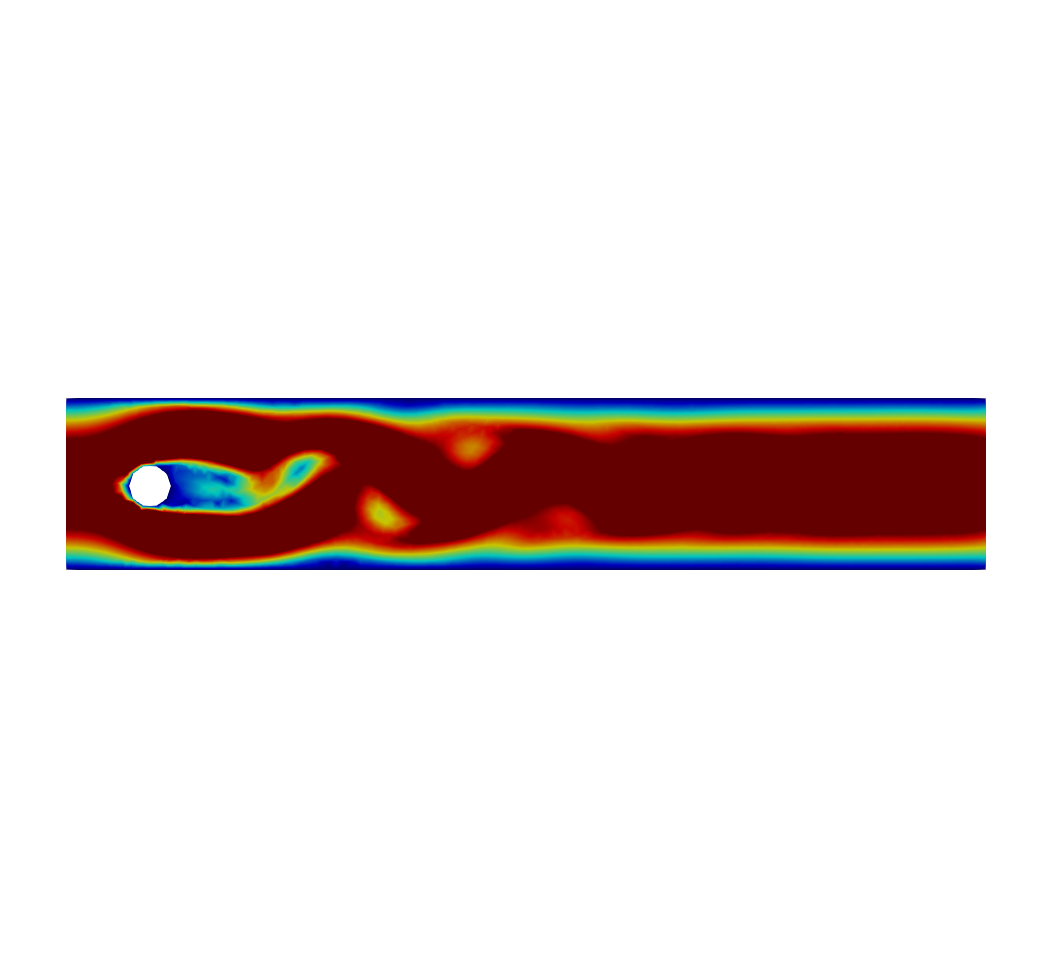}
      \includegraphics[width=1.0\textwidth,trim={0cm 12cm 0cm 10cm},clip]{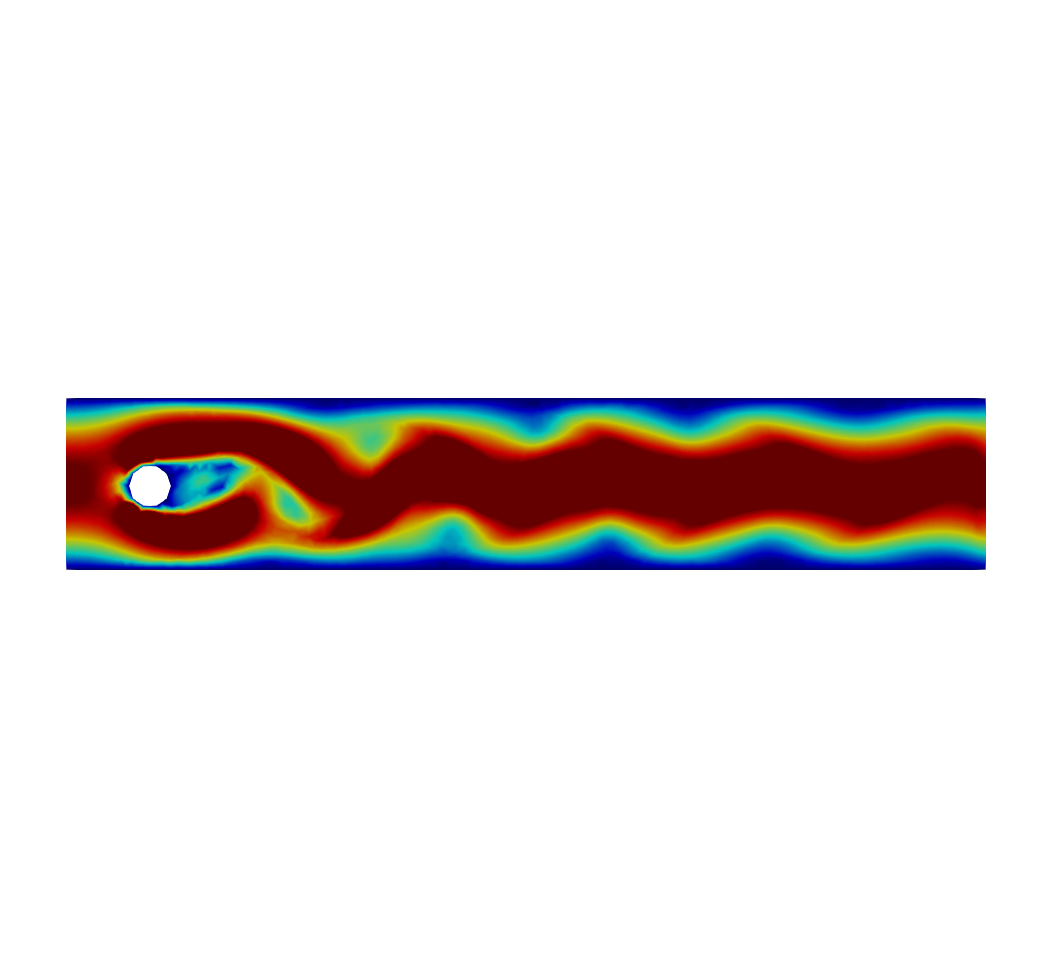}
      \vspace{4pt}
  \caption{$\textsf{dG1}$ when $\theta=1.0$: Contour of velocity magnitude when $t=2s, 3s, 5s, 6s$ from top to bottom}
  \label{cylinder_velocity_contour_ch1}
\end{figure}
\begin{figure}[ht]
\centering
\begin{subfigure}{0.4\textwidth}
  \centering
  \includegraphics[width=1.0\linewidth,trim={7cm 2cm 0cm 2cm},clip]{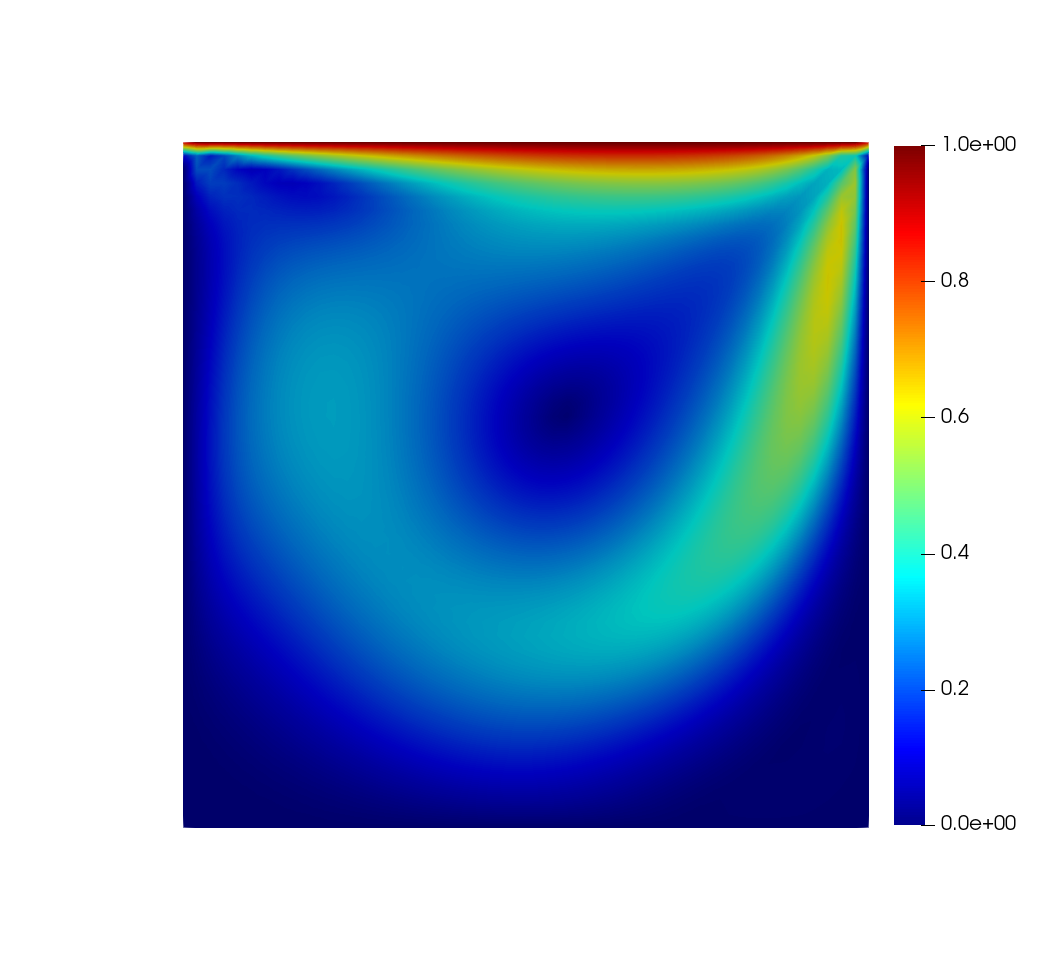}
\end{subfigure}%
\begin{subfigure}{.4\textwidth}
  \centering
  \includegraphics[width=1.0\linewidth,trim={7cm 2cm 0cm 2cm},clip]{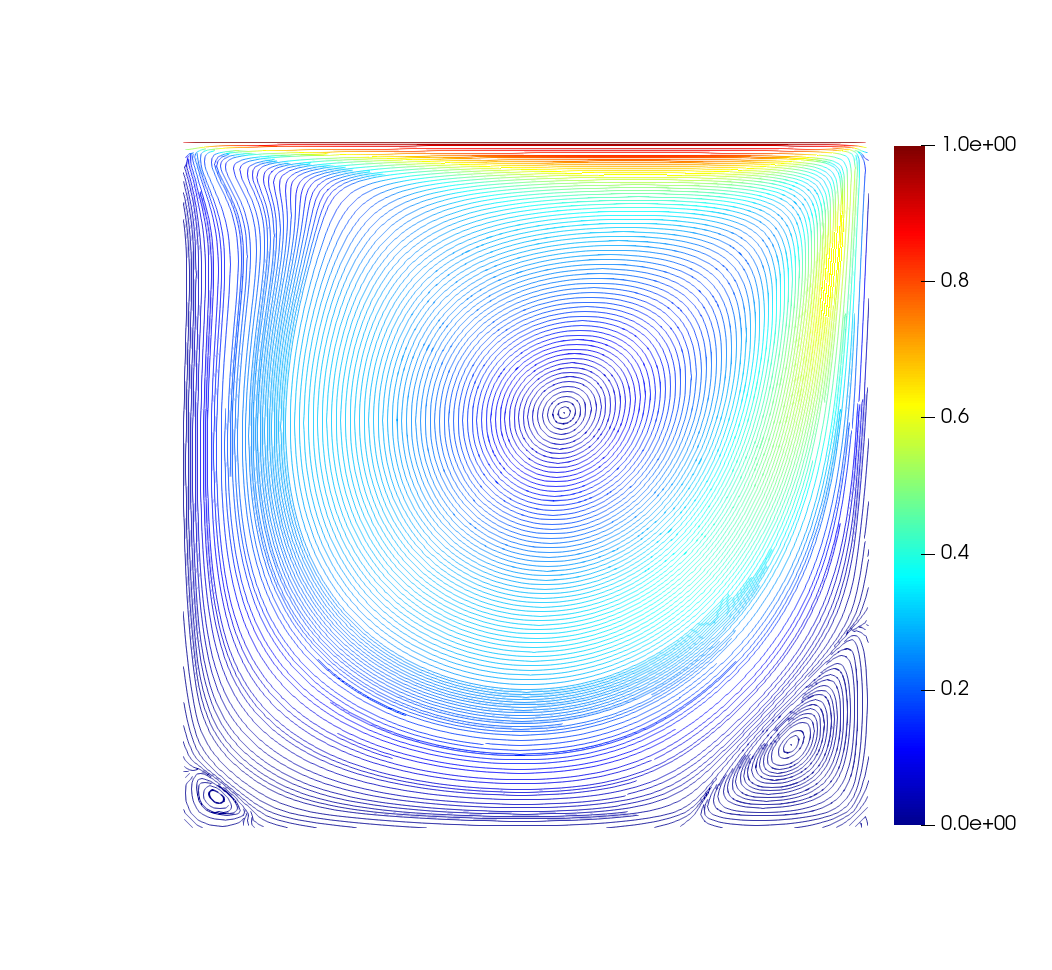}
\end{subfigure}
\caption{Lid Driven Flow: Contours of velocity magnitude (left) and stream trace (right) when $Re=400$, $k=3$, $h=0.0283$ and $\gamma=1000.0$.}
\label{fig:Lid_ch1}
\end{figure}

Figure~\ref{fig:Lid_ch1} shows that
\textsf{dG1} is able to capture the essential features of the lid driven flow. 

From Subsections \ref{dG1_1} and \ref{dG1_2}, we observe that the \emph{inconsistent} \textsf{dG1} formulation actually works well when choosing sufficiently large parameter $\gamma$. This fact was explained by our argument below equation \eqref{eqnZh}. In particular, experiments show that with properly chosen $\gamma$, one could still achieve high numerical accuracy and capture essential physical features.

\clearpage
\section{Conclusion}\label{conclusion}
We developed two classes of dG methods for the incompressible Euler and incompressible Navier-Stokes equations with the kinematic pressure, Bernoulli function and EMAC function with both semi- and fully discrete stability. When the velocity is $H(\text{div})$-conforming, we show that central flux conserves energy, linear momentum and angular momentum, while upwind flux conserves linear momentum and angular momentum under appropriate assumptions.
Numerical experiments were performed to demonstrate our findings, test the performances of the schemes, and compare with conventional schemes in the literature. When the velocity is $H(\text{div})$-conforming, the simulation results show that global conservation of physical quantities is not enough to guarantee the performance of the schemes, since they are not good indicators of the local behaviors. For the dG schemes, our schemes tend to achieve smaller absolute error when $k=0$ in Kovasznay flow, while achieve roughly the same order of accuracy otherwise in both the Taylor-Green vortex and Kovasznay flows when scheme \textsf{dG2} is considered. Temporal accuracy test shows that the Crank-Nicolson time discretization indeed has second order accuracy. Furthermore, we also show through
numerical examples that fairly accurate results for velocity could be obtained when the scheme \textsf{dG1} formulation is used with the Bernoulli function with a suitable penalization. The observation of locking in the Kovasznay flow indicates the important difference in the finite element schemes for steady and unsteady problems. Finally, with the problems of flow around a cylinder and lid driven flow, we show that \textsf{dG1} is able to capture the essential physics of the problem when the governing equations are with the Bernoulli function.
\section*{Data Availability Statement}

The data that support the findings of this study are available from the corresponding author upon reasonable request.
\clearpage

\providecommand{\bysame}{\leavevmode\hbox to3em{\hrulefill}\thinspace}
\providecommand{\MR}{\relax\ifhmode\unskip\space\fi MR }
\providecommand{\MRhref}[2]{%
  \href{http://www.ams.org/mathscinet-getitem?mr=#1}{#2}
}
\providecommand{\href}[2]{#2}

\end{document}